\documentclass{imsart}

\RequirePackage{amsthm,amsmath,amsfonts,amssymb}
\RequirePackage[numbers]{natbib}
\RequirePackage[colorlinks,citecolor=blue,urlcolor=blue]{hyperref}

\usepackage[capitalise]{cleveref}
\usepackage{color}
\usepackage{enumitem}

\startlocaldefs

\makeatletter
\newtheorem*{rep@theorem}{\rep@title}
\newcommand{\newreptheorem}[2]{%
\newenvironment{rep#1}[1]{%
 \def\rep@title{#2 \ref{##1}}%
 \begin{rep@theorem}}%
 {\end{rep@theorem}}}
\makeatother

\newreptheorem{theorem}{Theorem}
\newreptheorem{lemma}{Lemma}

\theoremstyle{plain}
\newtheorem{theorem}{Theorem}[section]
\newtheorem{lemma}[theorem]{Lemma}
\newtheorem*{lemma*}{Lemma}
\newtheorem{proposition}[theorem]{Proposition}
\newtheorem*{proposition*}{Proposition}
\newtheorem{corollary}[theorem]{Corollary}

\theoremstyle{remark}
\newtheorem{remark}[theorem]{Remark}
\newtheorem{example}[theorem]{Example}
\newtheorem{definition}[theorem]{Definition}
\newtheorem{assumption}[theorem]{Assumption}

\newcommand{\R}{\mathbb{R}}

\newcommand{\N}{\mathbb{N}}
\renewcommand{\Pr}[2][{}]{\mathbb{P}_{#1}\left(#2\right)}
\newcommand{\Prd}[2]{\mathbb{P}_#1\left(#2\right)}
\newcommand{\E}{\mathbb{E}}

\newcommand{\cupdu}[2]{
\underset{#1}{\overset{#2}{\bigcup }}}

\newcommand{\capdu}[2]{
\underset{#1}{\overset{#2}{\bigcap }}}

\newcommand{\sumdu}[2]{
\overset{#2}{\underset{#1}{\sum }}\,}

\newcommand{\proddu}[2]{
\overset{#2}{\underset{#1}{\prod }}\,}

\newcommand{\limn}[0]{
\underset{n\to \infty }{\lim }\;}

\newcommand\independent{\protect\mathpalette{\protect\independenT}{\perp}}
\def\independenT#1#2{\mathrel{\rlap{$#1#2$}\mkern2mu{#1#2}}}

\newcommand{\midvert}{\;\middle\vert\;}

\endlocaldefs

\begin{document}

\begin{frontmatter}

\title{Ranking-based rich-get-richer processes}
\runtitle{Ranking-based rich-get-richer processes}

\begin{aug}
\author[A]{\fnms{Pantelis P.} \snm{Analytis}},
\author[A]{\fnms{Alexandros} \snm{Gelastopoulos}}
\and
\author[B]{\fnms{Hrvoje} \snm{Stojic}}

\address[A]{Danish Institute for Advanced Study, University of Southern Denmark, Odense, Denmark}
\address[B]{Department of Economics and Business, Pompeu Fabra University, Barcelona, Spain}
\end{aug}

\begin{abstract}
We study a discrete-time Markov process $X_n\in\R^d$ for which the distribution of the future increments depends only on the relative ranking of its components (descending order by value). We endow the process with a rich-get-richer assumption and show that, together with a finite second moments assumption, it is enough to guarantee almost sure convergence of $X_n/n$. We characterize the possible limits if one is free to choose the initial state, and give a condition under which the initial state is irrelevant. Finally, we show how our framework can account for ranking-based P\'olya urns and can be used to study ranking-algorithms for web interfaces.
\end{abstract}


\end{frontmatter}

\section{Introduction}

Wealthy individuals tend to become even wealthier \citep{piketty2015capital}, popular websites become even more popular \citep{barabasi1999emergence}, and highly cited papers overshadow less cited ones, earning more future citations \citep{price1976general,redner1998popular}. Social and technological systems that preserve and amplify existing inequalities are said to be characterized by rich-get-richer dynamics \citep{merton1968matthew,price1976general,yule1925ii}. In these systems, initial conditions and randomness early in time drastically affect the course of future events---advantages obtained by agents early on are conserved and reinforced \citep{arthur1989competing,denrell2014perspective}. The above can result in socially objectionable outcomes, such as pervasive inequality in the distribution of wealth, and unfair outcomes where talented people or promising technologies cannot compete with already established ones \citep{page2006path}.

In many systems, and increasingly so in the online world, the rich-get-richer dynamics depend on the ranks of the various objects (people, options, institutions etc.) in terms of some quantity of interest. For example, companies or academic institutions might receive job applications based on some status ranking, which in turn can help these institutions retain their status by employing qualified individuals \citep{podolny1993status}. Similarly, scientists might submit their work to journals taking into account the journal's relative rank in terms of impact factor or some other metric, thus highly ranked journals are more likely to publish work of good quality and retain their position in the ranking \citep{hudson2013ranking,laband2013use}. Last but not least, users of online interfaces are more likely to click on entries that appear at the top of the screen, hence making these entries appear more relevant to other users \citep{joachims2005accurately,salganik2006experimental}. In all of these cases, it is the ranking of the different entities that confers an advantage to the more successful ones and thus drives the rich-get-richer dynamics. 

Although examples of systems characterized by ranking-based rich-get-richer dynamics abound, we still do not understand their dynamics and long-term behavior. There are only two previous relevant works, which have been developed in the context of P\'olya urns and can model ranking-based rich-get-richer systems as extreme cases. In the first such work Hill et al. \cite{hill1980strong} study the case of a P\'olya urn with balls of $d=2$ colors, one ball added at a time, and allow the probability of adding a red ball to be a function of the proportion of red balls. In other words, there is some function $f:[0,1]\to [0,1]$, such that the probability of the next ball being red is $f\left(X_n\slash n\right)$, where $X_n$ denotes the number of red balls at time $n$. If $f$ is taken to be constant in $\left[0,\frac{1}{2}\right) $ and in $\left(\frac{1}{2},1\right] $, then we get a ranking-based urn. In \cite{hill1980strong} it is shown that $X_n\slash n$ converges a.s., and then some results are given regarding the support of the limit (see also \cref{sectionPolyaUrn}). Importantly, a subset of the results in \citep{hill1980strong} allows a nowhere dense set of discontinuities for $f$, so they apply to the ranking-based case. It is not obvious though how to generalize these results to P\'olya urns with more colors or other types of processes.

The usual generalization to $d\in \N $ is to have the probability of adding a ball of color $i$ be proportional to a function of the count (or proportion) of balls of that color only, thus not allowing comparison of the counts of balls of different colors (for recent examples see \citep{chung2003generalizations,collevecchio2013preferential,laruelle2019nonlinear} - see also \citep{pemantle2007survey,zhu2009nonlinear} for surveys of results). A notable exception is the work of Arthur et al. \citep{arthur1986strong}, where the probabilities are allowed to depend on the whole vector of proportions of balls of each color. More precisely, there is an urn function
\begin{equation}
    f:\Delta ^{d-1}\to \Delta ^{d-1},\ \ \text{ where }\ \ \Delta ^{d-1}:= \left\{ x\in [0,1]^d,\ \sumdu{i}{}\,  x_i=1\right\},
\end{equation}
which takes as argument the vector of proportions of balls of each color, and its $i$-th component gives the probability of adding a ball of color $i$. The authors generalize some of the results in \citep{hill1980strong} to any $d\in \N$. In particular they show that under mild conditions on $f$ the process $X_n\slash n$ (where $X_n$ is now a vector) has positive probability of converging to any point $\theta \in \Delta ^{d-1}$ that is a \textit{stable} fixed point of $f$. According to the definition of stability used, in the ranking-based case all fixed points whose coordinates are all distinct are necessarily stable (see \cref{sectionPolyaUrn} for details). However, it is not claimed that the stable fixed points of $f$ are the only possible limits for $X_n\slash n$. Also, convergence of $X_n\slash n$ is shown only for certain special cases that do not cover ranking-based urns.

Even in the cases where the above results are applicable to ranking-based systems, their main limitation is that they are restricted to simple P\'olya-type processes, that is processes whose components increase one at a time and the increments are binary. But in many systems with ranking-dependent dynamics (e.g. journal impact factors, university ratings) the quantity of interest can take continuous values and the various components may change simultaneously. Given the paucity of mathematical work that can apply to systems with ranking-based rich-get-richer dynamics, especially for more general increments, our understanding of ranking-based processes remains limited.

In this work, we treat the problem in the context of 
(discrete-time) Markov processes, with an explicit dependence of the dynamics on the ranking. Specifically, we consider a non-homogeneous random walk in $\R^d$, for which the distribution of the steps depends only on the ranking of its components (descending order of their values). The fact that there are only finitely many possible rankings for a vector of $d$ components, and that the distribution of the jumps of the process does not change as long as the ranking doesn't change, allows us to consider separately the transitions between rankings and the dynamics when the ranking remains constant, the latter being nothing more than the dynamics of a sum of i.i.d. random vectors. Indeed, if the ranking converges to some limit value (i.e. eventually becomes constant), then a suitable application of the Strong Law of Large Numbers and the Central Limit Theorem gives us the behavior of $X^i_n\slash n$ in the limit (\cref{marketShareTheorem}). Therefore, the study of the long-term behavior of such processes is in large part a study of the long-term behavior of the ranking. This simplifies the study considerably and allows us to derive results under few assumptions. An essential assumption we make in order to show convergence is a type of a rich-get-richer condition, more precisely a ranking-based reinforcement condition (\cref{qualityOrderingAssumption}), and it is a weaker version of the following statement: conditioned on $X^i_n>X^j_n$, the difference $X^i_{n+1}-X^i_n$ has a larger mean than $X^j_{n+1}-X^j_n$.

Our results can be summarized as follows: under the above mentioned ranking-based reinforcement assumption and a finite second moments assumption, we show that in the limit $n\to \infty $ the ranking of the components of the process stops changing almost surely (\cref{settlingTheorem}). Moreover, we characterize the possible rankings in the limit (\cref{theoremTerminalRankings}). The latter result is independent of \cref{qualityOrderingAssumption}, but if this assumption holds, then we can characterize the possible limits for $X_n\slash n$ as well (\cref{marketShareTheorem}). By ``possible limit'' we mean that the probability of converging to this value is positive, for \textit{some} initial condition (distribution of $X_0$). \Cref{propositionLimitRankingsSufficient} gives a condition under which the probability of converging to any of the possible limits is positive for any initial condition. Next we specialize our results to the case of ranking-based P\'olya urns and relate them to the fixed points of the urn function (\cref{propositionPolya}), which allows a comparison with previous results. Even in this special case we get novel results regarding the limiting behavior of $X_n\slash n$. Finally, we describe an application to online rank-ordered interfaces.

\section{Main results}
\label{sectionResultsGeneral}

We begin by defining what we mean by ranking (\cref{sectionRanking}) and ranking-based processes (\cref{sectionFormulation}). \Cref{sectionSettling,sectionLimitRankings} contain our two main results: convergence of ranking and characterization of terminal rankings. In \cref{sectionLimitTheorems} we look at the limit behavior of the process $X_n$ itself and in \cref{sectionInitialDistributions} we consider the role of the initial condition.

\subsection{Rankings}
\label{sectionRanking}
For a finite set $S$, we denote by $|S|$ its cardinality and by $[S]=\{1,\ldots ,|S|\}$ the set of the first $|S|$ positive integers.

\begin{definition}
\label{definitionRanking}
Let $S$ be a finite set. A ranking of $S$ is a function $r:S\to [S]$ with the property that for each $a\in S$, \begin{equation}
\label{dummyEq67}
    card\{b\in S:r(b)<r(a)\}=r(a)-1.
\end{equation}
\end{definition}
We will say that $a$ is ranked higher than $b$ if $r(a)<r(b)$. \Cref{dummyEq67} requires that, for each $a\in S$, exactly $r(a)-1$ elements are ranked higher than $a$. Thus, we will call $r(a)$ the \textit{position} or \textit{rank} of $a$ in the ranking $r$. Note that two elements $a,b\in S$, $a\neq b$, can have the same position in $r$, that is we may have $r(a)=r(b)$. In this case we will say that these elements are equally ranked by $r$. In \cref{appendixWeakOrderings} we show that rankings of a set $S$ are equivalent to weak orderings on $S$.

Any bijection $r:S\to [S]$ satisfies \cref{dummyEq67}, hence it is a ranking. Such rankings will be called \textit{strict}. That is, strict rankings are such that no two elements of $S$ are equally ranked.

Given a vector $x=(x_1,\ldots ,x_d)\in \R ^d$, we denote by $rk(x)$ the unique ranking $r$ on the set $[d]=\{1,\ldots, d\}$ that satisfies $r(i)<r(j)$ if and only if $X_i>X_j$, for any $i,j\in [d]$. It is easy to check that there is indeed a unique such ranking, given by $r(i)=card\{j\in[d]:X_j>X_i\}+1$. The folk name for this map is the \textit{Standard Competition Ranking}.

We will denote by ${\cal R}={\cal R}(d)$ the set of all rankings of the set $[d]$.

\subsection{Ranking-based processes}
\label{sectionFormulation}

Let $(\Omega ,{\cal F},\mathbb{P})$ be a probability space and $\{{\cal F}_n\}_{n\in\N}$ a filtration on it. Let $\nu$ be a probability distribution on $\R^d$ with finite second moments, and for each $r\in {\cal R}(d)$ let $\mu^r$ be a probability distribution on $\R ^d$, also with finite second moments. We consider a time-homogeneous Markov process $X_n\in \R ^d$, adapted to $\{{\cal F}_n\}_n$, with initial distribution $\nu$ and with the law of its increments being $\mu ^r$, where $r$ is the current ranking. More precisely, the transition kernel $\mu$ is given by
\begin{equation}
\label{dummyEq19}
    \mu(x,B)=\mu^{rk(x)}(B-x), \ \ \ x\in \R^d,\ B\in {\cal B}(\R^d),
\end{equation}
where ${\cal B}(\R^d)$ denotes the Borel $\sigma $-algebra of $\R^d$ and $B-x=\{ y\in \R^d:y+x\in B\} $ denotes the translation of $B$ by the vector $-x$. We will call such a process a ($d$-dimensional) \textit{ranking-based process}.

\Cref{dummyEq19} implies that for any $B\in {\cal B}(\R^d)$,
\begin{equation}
\label{eqCondInd1}
    \Pr{\Delta X_{n+1}\in B\midvert {\cal F}_n}=\mu^{rk(X_n)}(B)\ \ a.s.,
\end{equation}
where $\Delta X_{n+1}=X_{n+1}-X_n$. In particular, the process is space-homogeneous within subsets of $\R^d$ that correspond to a fixed ranking, that is subsets of the form $\{x\in \R^d:rk(x)=r\}$, for $r\in {\cal R}(d)$ (but it is not space-homogeneous in general). \Cref{eqCondInd1} also implies that, conditioned on the ranking at time $n$, $\Delta X_{n+1}$ is independent of ${\cal F}_n$, that is
\begin{equation}
\label{eqCondInd2}
\Delta X_{n+1}\underset{rk(X_n)}{\independent} {\cal F}_n.
\end{equation}

We will use the shorthand notation $\mu^r(x_i>x_j)$ to mean $\mu^r(\{x\in\R^d:x_i\neq x_j\})$ and similarly for other events that involve comparisons of components of $x$.

For each $r\in{\cal R}$, we denote by $Z^r=(Z^r_1,\ldots ,Z^r_d)$ a random variable with distribution $\mu ^r$. This will be especially useful when considering differences of the form $\Delta X^i_{n+1}-\Delta X^j_{n+1}$, whose distribution cannot be directly expressed via $\mu^r$. Note that conditioned on $rk(X_n)$, $\Delta X^i_{n+1}-\Delta X^j_{n+1}$ has the distribution of $Z_i^{rk(X_n)}-Z_j^{rk(X_n)}$, that is, for each $B\in{\cal B}(\R^d)$
\begin{equation}
\label{eqZ}
    \Pr{\Delta X^i_{n+1}-\Delta X^j_{n+1}\in B\midvert rk(X_n)}=\Pr{Z_i^{rk(X_n)}-Z_j^{rk(X_n)}\in B}.
\end{equation}

We denote by $q^r_i$ the mean and by $\sigma ^r_i$ the standard deviation of the $i$-th component of the distribution $\mu ^r$, that is $q^r_i=\E\left[Z^r_i\right]$ and $\sigma ^r_i=\sqrt{Var(Z^r_i)}$. We will also use the vector notation $q^r=(q^r_1,\ldots ,q^r_d)$ for the mean.

For the rest of the paper, we fix $d\in\N$ and a $d$-dimensional ranking based process $X_n$ adapted to $\{{\cal F}_n\}_n$, with the associated $\mu ^r$'s, $\mu$, $Z^r$'s, $q^r_i$'s and $\sigma ^r_i$'s. Strictly speaking, the Markov process $X$ is described by the pair $(\nu, \mu )$. However, we will often abuse terminology and talk about a single process $X$ while allowing the initial distribution $\nu$ to vary. We will use the notation $\mathbb{P}_{\nu}$ for probabilities of events that depend on the initial distribution. The subscript $\nu$ will often be omitted for expressions that do not depend on the initial distribution (as in \cref{eqZ}). Both the distributions $\mu ^r$ and initial distribution $\nu$ will always be assumed to have finite second moments.

We will also suppress the integer $d$ in the notation for the set of all rankings of $[d]$ and write ${\cal R}={\cal R}(d)$.

\subsection{Convergence of ranking}
\label{sectionSettling}

As $n\to \infty$, the ranking of $X_n$ may keep changing or it might converge to some particular ranking $r\in {\cal R}$ (where ${\cal R}$ is endowed with the discrete topology). We have the following definition.

\begin{definition}
\label{defSettling2}
Let $X$ be a ranking-based process. We say that $rk(X_n)$ converges to $r \in {\cal R}$ and write $rk(X_n)\to r$ (or $\limn rk(X_n)=r$) , if $rk(X_n)=r$ for all sufficiently large $n$, that is
\begin{equation}
    \{rk(X_n)\to r\}=\cupdu{n_0\in\N}{}\,\capdu{n=n_0}{\infty}\{rk(X_n)=r\}=\underset{n}{\liminf}\{rk(X_n)=r\} .
\end{equation}
We say that a ranking $r \in {\cal R}$ is terminal (for the transition kernel $\mu$), if there exists some initial distribution $\nu$, such that
\begin{equation}
\label{eqDefinitionTerminal}
    \Prd{\nu}{rk(X_n)\to r}>0.
\end{equation}
Otherwise, we say that $r$ is transient.
\end{definition}

Knowing that the ranking converges is useful, because then we can predict the long-term behavior of the process (see \cref{sectionLimitTheorems}). We will therefore seek conditions under which the ranking is guaranteed to converge. 

As a first step, we ask the following question: if we know that $X^i_{n_0}>X^j_{n_0}$ occurs for some $n_0\in\N$, is it likely that $X^i_n>X^j_n$ for all $n>n_0$? The following definition and proposition give a sufficient condition for the probability of this event to be positive and bounded away from zero.

\begin{definition}
\label{defDominance2}
Let $\{X_n\}_{n\in\N}$ be a $d$-dimensional ranking-based process with the associated distributions $\mu^r$ and means $q^r_1,\ldots,q^r_d$, and let $i,j\in [d]$.
\begin{itemize}
    \item We say that $i$ quasi-dominates $j$, if for any ranking $r$ such that $r(i)<r(j)$ we have either $q^r_i>q^r_j$ or $\mu^r(x_i\neq x_j)=0$.
    \item We say that $i$ dominates $j$ if we further have that for any ranking $r$ such that $r(i)=r(j)$, either $\mu^r(x_i>x_j)>0$ or $\mu^r(x_i\neq x_j)=0$.
\end{itemize}
\end{definition}

Note the relation between quasi-dominance and the (loosely defined) concept of rich-get-richer dynamics: if $i$ quasi-dominates $j$, then $X^i_n$ increases on average faster than $X^j_n$ whenever it is already larger (or they vary in exactly the same way). The extra condition for dominance says that $X^i_n$ has a non-zero probability of passing ahead after a tie (or, again, the two components vary in exactly the same way).

\begin{proposition}
\label{lemmaProbaNeverChange}
Let $i,j\in [d]$.
\begin{enumerate}
    \item If $i$ quasi-dominates $j$, then there exists some $\epsilon>0$ such that for any initial distribution $\nu$ and any ${\cal F}_n$-stopping time $s$, we have
    \begin{equation}
    \label{dummyEq75}
        \Prd{\nu}{\capdu{n=s}{\infty }\left\{ X^{i}_{n}>X^{j}_{n}\right\}\midvert {\cal F}_{s}} \geq \epsilon\ \  \text{ a.s. on }\{s<\infty \}\cap \{X^i_{s}>X^j_{s}\}.
    \end{equation}
    \item If $i$ dominates $j$, we further have
    \begin{equation}
        \Prd{\nu}{\capdu{n=s}{\infty }\left\{ X^{i}_{n}\geq X^{j}_{n}\right\}\midvert {\cal F}_{s}} \geq \epsilon\ \  \text{ a.s. on }\{s<\infty \}\cap \{X^i_{s}\geq X^j_{s}\}.
    \end{equation}
\end{enumerate}
\end{proposition}

For a concrete case, if we take the a.s. constant stopping time $s=n_0$, then \cref{dummyEq75} implies in particular that
\begin{equation}
    \Pr[\nu]{\capdu{n=n_0}{\infty}\{X^i_{n}>X^j_{n}\}\midvert X^i_{n_0}>X^j_{n_0}}\geq \epsilon ,
\end{equation}
whenever the expression on the left hand side makes sense (i.e. whenever $\Pr[\nu]{X^i_{n_0}>X^j_{n_0}}>0$).

We postpone the proof in order to get to our main result for this section. For ease of reference we state the condition for that theorem separately:

\begin{assumption}[Ranking-based reinforcement]
\label{qualityOrderingAssumption}
    For any pair of indices $i,j\in [d]$, either one of them dominates the other, or they quasi-dominate each other.
\end{assumption}

Note that it is possible for both $i$ and $j$ to dominate each other; the above assumption would still be satisfied. This means that the ``dominance'' relation does not have to be trichotomous. It does not have to be transitive either. However, a transitive trichotomous relation (i.e. a strict total order) on $[d]$ would satisfy \cref{qualityOrderingAssumption}.

\begin{example}
\label{example:additivePolyaUrn}
Let $X_n=(X^1_n,\ldots X^d_n)$ give the number of balls of each of $d$ colors in an urn. At each time step, a single ball is added, with probabilities for each color depending on the ranking. Note that in this case $q^r_i$ is equal to the probability of adding a ball of color $i$ when the ranking is $r$ (see also first paragraph of \cref{sectionPolyaUrn}). These probabilities will be determined as follows: Each color has a propensity $a_i\geq 0$ to be chosen. Moreover, there are real numbers $\lambda _1>\cdots >\lambda _d\geq 0$, with $\lambda _i$ denoting an additive bonus to the propensity of the color(s) currently ranked $i$-th. More specifically, the probability of adding a ball of color $i$, given that the current ranking is $r$, is
\begin{equation}
\label{eqDummy51}
    q^r_i=\frac{a_i+\lambda _{r(i)}}{\sumdu{j=1}{d}\left(a_j+\lambda _{r(j)}\right)}.
\end{equation}
We claim that this process satisfies \cref{qualityOrderingAssumption}. To see this, let $i,j\in [d]$ and suppose without loss of generality that $a_i\geq a_j$. We have the following cases:
\begin{itemize}
    \item $a_i=a_j$: By \cref{eqDummy51} we have that $q^r_i>q^r_j$ whenever $i$ is ranked higher than $j$ and vice versa. That is, $i$ and $j$ quasi-dominate each other.
    \item $a_i>a_j$: We similarly get that color $i$ quasi-dominates color $j$. Moreover, when $i$ and $j$ are ranked equally (i.e. $r(i)=r(j)$), \cref{eqDummy51} gives $q^r_i>q^r_j$, that is it is more likely for color $i$ to be chosen. This shows that $i$ dominates $j$.
\end{itemize}
Thus our claim is proved.
\end{example}

We now state and prove our main theorem for this section.

\begin{theorem}[Convergence of ranking]
\label{settlingTheorem}
Let $X_n$ be a ranking-based process satisfying \cref{qualityOrderingAssumption}. Then, $rk(X_n)$ converges a.s., for any initial distribution $\nu$.
\end{theorem}

\begin{proof}
It is enough to show that for each pair of indices $i,j$, the relative ranking of $X^i_n$ and $X^j_n$ eventually stops changing with probability $1$. So let $i\neq j$ and, without loss of generality, assume that $i$ quasi-dominates $j$ (see \cref{qualityOrderingAssumption}). Define $s_0=0$ and inductively $t_m=\inf \{ n>s_{m-1}:X^i_n>X^j_n\} $ and $s_m=\inf \{ n>t_m:X^i_n\leq X^j_n\} $. Notice that $\{s_m=\infty\}=\capdu{n=t_m}{\infty }\{X^i_n>X^j_n\}$. Therefore, \cref{lemmaProbaNeverChange} applied to $s=t_m$ implies that there exists some $\epsilon$, not depending on $m$, such that
\begin{equation}
\label{dummyEq46}
\begin{aligned}
    \Prd{\nu}{s_m=\infty \midvert {\cal F}_{t_m}} & \geq \epsilon>0,\ \ a.s.
\end{aligned}
\end{equation}
on $\{t_m<\infty,X^i_{t_m}>X^j_{t_m}\}=\{t_m<\infty \}$. In particular, if $\Pr[\nu]{t_m<\infty }>0$, then
\begin{equation}
    \Pr[\nu]{s_m=\infty\midvert t_m<\infty } \geq \epsilon
\end{equation}
and
\begin{equation}
\label{eqDummy49}
\begin{aligned}
    \Prd{\nu}{s_m<\infty } & =\Prd{\nu}{\{ t_m<\infty \} \cap \{ s_m<\infty \}} \\ &
    =\Prd{\nu}{t_m<\infty}\cdot \Prd{\nu}{s_m<\infty \midvert t_m<\infty \}}\\
    & \leq (1-\epsilon )\cdot \Prd{\nu}{t_m<\infty} \\
    & \leq (1-\epsilon )\cdot \Prd{\nu}{ s_{m-1}<\infty}.
\end{aligned}
\end{equation}
Although we have assumed $\Pr[\nu]{t_m<\infty }>0$, \cref{eqDummy49} continues to hold even if $\Pr[\nu]{t_m<\infty }=0$, because then $\Pr[\nu]{s_m<\infty }=0$ as well.

By \cref{eqDummy49} and induction we have $\Prd{\nu}{s_m<\infty }\leq \left(1-\epsilon \right) ^m$, therefore
\begin{equation}
    \Pr[\nu]{\capdu{m\in\N}{\infty }\left\{ s_m<\infty \right\}}=0. 
\end{equation}

Hence, with probability $1$, either $X^i_n\leq X^j_n$ finitely often (henceforth abbreviated f.o.) or $X^i_n>X^j_n$ f.o. If $X^i_n\leq X^j_n$ f.o., then $X^i_n>X^j_n$ for all sufficiently large $n$, so we are done. Now assume that $X^i_n>X^j_n$ f.o. and separate two cases, according to \cref{qualityOrderingAssumption}:
\begin{itemize}
    \item $j$ also quasi-dominates $i$: We get similarly that either $X^i_n\geq X^j_n$ f.o. or $X^i_n<X^j_n$ f.o. As before, in the first case we are done. In the second case, we have both $X^i_n<X^j_n$ and $X^i_n>X^j_n$ f.o., so that $X^i_n=X^j_n$ for all sufficiently large $n$.
    \item $i$ dominates $j$: Using the second part of \cref{lemmaProbaNeverChange} we get that either $X^i_n<X^j_n$ f.o. or $X^i_n\geq X^j_n$ f.o. The situation is identical as in the first case.
\end{itemize}
\end{proof}

We now turn to the proof of \cref{lemmaProbaNeverChange}. We will need the following lemma, which generalizes a property of biased random walks to the case that the transition probabilities are not constant, but vary in a finite set. Its proof is given in the Appendix. A related result is obtained in \citep[Th. 2.5.12]{menshikov2016non} by different methods.

\begin{lemma}
\label{lemmaBiasedRandomWalk}
Let $(\Omega , {\cal G},\mathbb{P})$ be a probability space. Let $S$ be a finite set and for each $r\in S$, $\nu^r$ a distribution on $\R$ such that it either has positive mean or $\nu^r(\{0\})=1$. Let $\{ R_n\} _{n\in\N}$ be a sequence of random elements in $S$ and $\{ Y_n\} _{n\in \N }$ a sequence of random variables with $Y_0=0$. Suppose that $\Delta Y_{n+1}$ is conditionally independent of $\{(Y_k,R_k)\}_{k\leq n}$ conditioned on $R_n$, with distribution $\nu ^{R_n}$. In other words, for any $A\in {\cal B}(\R)$, $n\in\N$,
\begin{equation}
\label{eqDummy93}
    \Pr{\Delta Y_{n+1}\in A\midvert \{(Y_k,R_k)\}_{k\leq n}}=\nu^{R_n}(A)\ \ a.s.
\end{equation}
Then,
\begin{equation}
    \Pr{\capdu{n\in\N}{}\{Y_n\geq 0\}}\geq \epsilon >0,
\end{equation}
where $\epsilon$ depends only on the distributions $\nu^r$, $r\in S$.
\end{lemma}

We note that if $|S|=1$, then \cref{lemmaBiasedRandomWalk} reduces to the well-known result that a biased one-dimensional random walk with positive mean has positive probability of never admitting negative values (see \cite[Corollary 9.17]{kallenberg2006foundations}).

\begin{proof}[Proof of \cref{lemmaProbaNeverChange}]
\begin{enumerate}
    \item Let $s$ and $\nu $ be given and define $\tau = \min\{n\geq s:X^i_n\leq X^j_n\}$ and
    \begin{equation}
        Y_n=X^i_{\tau \wedge n}-X^j_{\tau \wedge n}
    \end{equation}
    Note that $Y_n>0$ for all $n\geq s$ implies $X^i_n>X^j_n$ for all $n\geq s$. Therefore, it is enough to show that, for some $\epsilon >0$ that does not depend on $s$ or $\nu$,
    \begin{equation}
    \label{eqDummy35}
    \Prd{\nu}{\capdu{n=s}{\infty }\left\{ Y_{n}>0\right\}\midvert {\cal F}_{s}} \geq \epsilon \ \ \text{ a.s. on }\{Y_{s}>0\}.
    \end{equation}
    We have
    \begin{equation}
    \label{eqDummy37}
        \Delta Y_{n+1}=\mathbf{1}_{\tau >n}\cdot (\Delta X^i_{n+1}-\Delta X^j_{n+1}),
    \end{equation}
    where $\mathbf{1}_A$ denotes the indicator function of the set $A$. It follows that conditioned on $rk(X_n)$ and $\mathbf{1}_{\tau >n}$, $\Delta Y_{n+1}$ is independent of ${\cal F}_n$ (see \cref{eqCondInd1}). Moreover, its conditional distribution is equal to that of $Z^{rk(X_n)}_i-Z^{rk(X_n)}_j$ in the case $\tau >n$ (by \cref{eqZ}), while $\Delta Y_{n+1}=0$  identically otherwise.

    Let $F\in {\cal F}_s$ be any event with $\Pr{F}>0$ and consider the probability measure $\Pr[\nu,F]{\cdot }=\Pr[\nu]{\cdot\midvert F}$. We apply \cref{lemmaBiasedRandomWalk} for this measure and the sequence $\{Y_{n+s}-Y_{s}\}_{n\in\N}$, with $S={\cal R}\sqcup \{\alpha \}$ (where $\alpha$ is an arbitrary new element) and
    \begin{equation}
    R_n=
    \begin{cases}
        rk(X_{n+s}), & \text{ if } \tau>n+s,\\
        \alpha, & \text{ otherwise,}
    \end{cases}
    \end{equation}
    The distributions $\nu ^r$ in \cref{lemmaBiasedRandomWalk} are equal to the distributions of $Z^r_i-Z^r_j$ for $r\in {\cal R}$, while $\nu ^{\alpha }$ is the singular probability measure satisfying $\nu^{\alpha}(\{0\})=1$. \Cref{lemmaBiasedRandomWalk} thus gives
    \begin{equation}
    \Prd{\nu}{\capdu{n\in\N}{}\left\{ Y_{n+s}-Y_{s}\geq 0\right\}\midvert F} \geq \epsilon>0,
    \end{equation}
    where $\epsilon $ depends only on the $\mu ^r$'s (distributions of $Z^r$'s). Since $F$ was arbitrary, we get
    \begin{equation}
    \Prd{\nu}{\capdu{n\in\N}{}\left\{ Y_{n+s}-Y_{s}\geq 0\right\}\midvert {\cal F}_{s}} \geq \epsilon\ \ a.s.,
    \end{equation}
    from where \Cref{eqDummy35} follows.

\item Let $\tau '=\inf\, \left\{n\geq s: X^i_n\neq X^j_n\right\}$. We may assume that $X^i_s=X^j_s$ and $\tau'<\infty $, since on $\{X^i_s>X^j_s\}$ part (a) applies, while on $\{\tau'=\infty \}$ the result holds trivially. On $\{X^i_s=X^j_s,\tau'<\infty \}$ we have $\capdu{n=s}{\infty }\left\{ X^{i}_{n}\geq X^{j}_{n}\right\}=\capdu{n=\tau '}{\infty }\left\{ X^{i}_{n}\geq X^{j}_{n}\right\}$ and $\tau'\geq s+1$, it is therefore enough to show that
\begin{equation}
\label{eqDummy71}
    \Pr[\nu]{\capdu{n=\tau'}{\infty }\left\{ X^{i}_{n}> X^{j}_{n}\right\}\midvert {\cal F}_{\tau'-1}}\geq \epsilon\ \ a.s.
\end{equation}
or, by part (a),
\begin{equation}
\label{eqDummy73}
    \Pr[\nu]{X^{i}_{\tau'}> X^{j}_{\tau'}\midvert {\cal F}_{\tau'-1}}\geq \epsilon'\ \ a.s.,
\end{equation}
for some $\epsilon'>0$ that does not depend on $\nu $ or $s$.

Let $R'=\{r\in {\cal R}:r(i)=r(j),\mu^r(x_i>x_j)>0\}$ be the subset of rankings that rank $i$ and $j$ equally, but they give positive probability to $i$ to pass ahead on the next step. Since by assumption all other rankings with $r(i)=r(j)$ satisfy $\mu^r(x_i\neq x_j)=0$, $rk(X_n)$ must take a value in $R'$ before we can have $X^i_n\neq X^j_n$. That is, $rk(X_{\tau '-1})\in R'$ a.s., hence also
\begin{equation}
\begin{aligned}
    \Pr{X^i_{\tau'}>X^j_{\tau'}\midvert {\cal F}_{\tau'-1}} & =\Pr{\Delta X^i_{\tau'}>\Delta X^j_{\tau'}\midvert {\cal F}_{\tau'-1}}\\
    & =\mu^{rk(X_{\tau'-1})}(x_i>x_j)\\
    & \geq \underset{r\in R'}{\min }\, \mu^r(x_i>x_j)>0\ \ a.s.,
\end{aligned}
\end{equation}
where the second equality follows from \cref{eqCondInd1}.
\end{enumerate}
\end{proof}

\subsection{Terminal rankings}
\label{sectionLimitRankings}

\cref{settlingTheorem} says that \cref{qualityOrderingAssumption} guarantees convergence of $rk(X_n)$, but it doesn't say anything about the possible limits. In this section we deal with the question of what the possible limit rankings are. Recall that a ranking is terminal if $\Prd{\nu}{rk(X_n)\to r}>0$ for some probability distribution $\nu$ (\cref{defSettling2}). Our main result in this section is the following:

\begin{theorem}[Terminal rankings]
\label{theoremTerminalRankings}
Let $X_n$ be a $d$-dimensional ranking-based process with the associated distributions $\mu^r$ and means $q^r_1,\ldots ,q^r_d$. A ranking $r$ is terminal if and only if, for any $i,j\in [d]$:
\begin{itemize}
    \item If $r(i)=r(j)$ then $\mu^r(x_i\neq x_j)=0$.
    \item If $r(i)<r(j)$ then either $q^r_i>q^r_j$ or $\mu^r(x_i\neq x_j)=0$.
\end{itemize}
\end{theorem}

Let us give some intuition behind \cref{theoremTerminalRankings}. If $rk(X_n)\to r$, then there exists some $n_0\in \N $ such that $rk(X_n)=r$ for all $n\geq n_0$, so $\Delta X_{n+1}$ is distributed according to $\mu ^r$ for all $n\geq n_0$. In particular, for any $i\in [d]$, the $\Delta X^i_{n+1}$'s behave like i.i.d. random variables with mean $q^r_i$ and finite variance, hence $X^i_{n}\slash n\to q^r_i$ (see also \cref{marketShareTheorem}). Therefore, if $r$ ranks $i$ higher than $j$, for the ranking to remain equal to $r$, we must have that $q^r_i>q^r_j$. Note that in particular $q^r_i=q^r_j$ is not enough. An exception to the latter is if $\Delta X^i_{n+1}=\Delta X^j_{n+1}$ a.s. (equivalently $\mu^r(x_i\neq x_j)=0$), so that the two components change in exactly the same way. On the other hand, if $i$ and $j$ are ranked equally, then we must necessarily have $\Delta X^i_{n+1}=\Delta X^j_{n+1}$ a.s. for the ranking not to change. The above theorem says that these conditions are not only necessary, but also sufficient for the ranking to have a positive probability to remain the same for all $n\geq n_0$.


\Cref{theoremTerminalRankings} characterizes all terminal rankings by an easy to check criterion. Note that it does not require \cref{qualityOrderingAssumption}. However, without that assumption $rk(X_n)$ is not guaranteed to converge (see \cref{settlingTheorem}). Also note that even if we know that $r$ is terminal, we don't know whether $\Prd{\nu}{rk(X_n)\to r}>0$ for a \textit{specific} initial distribution $\nu$. This is the topic of \cref{sectionInitialDistributions} (see in particular \cref{propositionLimitRankingsSufficient}).

If we can exclude the case $\mu^r(x_i\neq x_j)=0$, then we get the following simplification of \cref{theoremTerminalRankings}. 

\begin{corollary}
\label{theoremTerminalRankingsSpecial}
Suppose that $\mu^r(x_i\neq x_j)>0$ for all $i,j\in [d]$ and all $r\in {\cal R}$. Then, a ranking $r$ is terminal if and only if it is a strict ranking and
\begin{equation}
\label{eqDescendingMeans}
    q^{r}_{r^{-1}(1)}>q^{r}_{r^{-1}(2)}>\ldots >q^{r}_{r^{-1}(d)},
\end{equation}
where $r^{-1}$ denotes the inverse of $r$.
\end{corollary}
\begin{proof}
The case $\mu^r(x_i\neq x_j)=0$ is excluded by assumption, so by \cref{theoremTerminalRankings} a ranking is terminal if and only if for any $i,j$ with $r(i)<r(j)$ we have $q^r_i>q^r_j$, or equivalently, if for any $i<j$, $q^r_{r^{-1}(i)}>q^r_{r^{-1}(j)}$.
\end{proof}

For the proof of \cref{theoremTerminalRankings} we are going to need a construction that will also be used again later on. Specifically, given a ranking-based process $X_n$ and a ranking $r\in {\cal R}$, we construct another process $Y_n$ that is identical to $X_n$ up to some point $n_0\in\N$, and it has i.i.d. increments afterwards with distribution $\mu^r$. It has the additional property that it remains equal to $X_n$ as long as their common ranking remains equal to $r$. The benefit of this is that we can work with the simpler process $Y_n$ and then transfer results to $X_n$.

\begin{lemma}
\label{lemmaMirroring}
For any $r\in {\cal R}$ and any $n_0\in\N$, there exists a process $Y_n\in\R^d$ and a filtration ${\cal G}_n\supset {\cal F}_n$ such that:
\begin{enumerate}[label=\roman*.]
    \item $Y_n=X_n$ for all $n\leq n_0$.
    \item $\{\Delta Y_{n}\}_{n\geq n_0+1}$ is a sequence of i.i.d. random vectors with distribution $\mu ^r$. Moreover, $Y_n\in {\cal G}_n$ for each $n\in\N$, and $\Delta Y_{n+1}\independent {\cal G}_{n}$ for each $n\geq n_0$.
    \item For any $n>n_0$, on both $\capdu{k=n_0}{n-1}\{rk(X_k)=r\}$ and $\capdu{k=n_0}{n-1}\{rk(Y_k)=r\}$ we have $Y_k=X_k$ a.s. for $k=0,1,\ldots ,n$. In particular, on both $\capdu{k=n_0}{\infty }\{rk(X_k)=r\}$ and $\capdu{k=n_0}{\infty }\{rk(Y_k)=r\}$ we have $Y_n=X_n$ a.s. for all $n\in\N$.
\end{enumerate}
\end{lemma}

A process $Y_n$ that satisfies the above properties (for some filtration ${\cal G}_n$) will be said to \textit{$(r,n_0)$-mimic $X_n$}.
\begin{proof}
Let $\{U_n\}_{n\in\N}$ be a sequence of i.i.d random vectors in $\R ^d$ with distribution $\mu ^r$, independent of ${\cal F}_{\infty}$, and let ${\cal G}_n=\sigma (U_1,\ldots ,U_n,{\cal F}_n)$ and $\tau =\min \{ n\geq n_0:rk(X_n)\neq r\} $. Define
\begin{equation}
\label{eqRmirror}
    Y_n=X_{\tau \wedge n}+\sumdu{m=\tau +1}{n}U_m.
\end{equation}
with the convention that the sum is $0$ if $\tau+1>n$. Property (i) follows from the fact that $\tau \geq n_0$. For property (ii), note that since $\{\tau \leq n\}$ is ${\cal F}_n$-measurable, we get $Y_n\in {\cal G}_n$. Moreover,
\begin{equation}
    \Delta Y_{n+1}=\mathbf{1}_{\tau >n}\cdot \Delta X_{n+1}+\mathbf{1}_{\tau \leq n}\cdot U_{n+1},
\end{equation}
In particular, for any $n\geq n_0$ and any $S\in{\cal B}(\R^d)$, on $\{\tau >n\}$ we have
\begin{equation}
    \Pr[\nu]{\Delta Y_{n+1}\in S\midvert {\cal G}_n}=\Pr[\nu]{\Delta X_{n+1}\in S\midvert {\cal G}_n}=\mu ^{rk(X_n)}(S)=\mu ^r(S)\ \ a.s.,
\end{equation}
where the second equality follows from \cref{eqCondInd1}. Also, on $\{\tau \leq n\}$ we have
\begin{equation}
    \Pr[\nu]{\Delta Y_{n+1}\in S\midvert {\cal G}_n}=\Pr[\nu]{U_{n+1}\in S\midvert {\cal G}_n}=\Pr[\nu]{U_{n+1}\in S}=\mu ^r(S)\ \ a.s.
\end{equation}
Combining the last two equations we get
\begin{equation}
    \Pr[\nu]{\Delta Y_{n+1}\in S\midvert {\cal G}_n}=\mu ^r(S)\ \ a.s.,\ \ n\geq n_0,\ S\in{\cal B}(\R^d).
\end{equation}
Therefore, the sequence $\{\Delta Y_{n+1}\}_{n\geq n_0}$ is i.i.d. and, for each $n\geq n_0$, $\Delta Y_{n+1}$ has distribution $\mu ^r$ and is independent of ${\cal G}_n$, which completes the proof of (ii).

For property (iii), let $m\leq n$ and note that on the set $\{Y_m\neq X_m\}$ we have $\tau <m<\infty$ and by definition $\tau\geq n_0$, $Y_{\tau}=X_{\tau}$, and $rk(Y_{\tau })=rk(X_{\tau })\neq r$. Therefore, the intersection of $\{Y_m\neq X_m\}$ with both $\capdu{k=n_0}{n-1}\{rk(X_k)=r\}$ and $\capdu{k=n_0}{n-1}\{rk(Y_k)=r\}$ is empty a.s.
\end{proof}

\begin{proof}[Proof of necessity for \cref{theoremTerminalRankings}]
Let $r$ be a terminal ranking. Then, there exists some initial distribution $\nu$ and some $n_0\in \N$ such that $\Prd{\nu}{A}>0$, where
\begin{equation}
    A=\capdu{n=n_0}{\infty }\{rk(X_n)=r\}.
\end{equation}
Let $Y_n$ $(r,n_0)$-mimic $X_n$. By \cref{lemmaMirroring}iii we have
\begin{equation}
    \capdu{n=n_0}{\infty }\{rk(Y_n)=r\}=\capdu{n=n_0}{\infty }\{rk(X_n)=r\}=A.
\end{equation}
Fix some $i,j\in [d]$ and note that the sequence $d_n=Y^i_n-Y^j_n$, $n\geq n_0$, performs a random walk, starting at $d_{n_0}=Y^i_{n_0}-Y^j_{n_0}=X^i_{n_0}-X^j_{n_0}$, and with the step $\Delta d_{n+1}=d_{n+1}-d_n$ having the same distribution as $Z^r_i-Z^r_j$ (see \cref{eqZ}). In particular, for any $n\geq n_0$,
\begin{equation}
\label{dummyEq56}
    \Pr[\nu]{\Delta d_{n+1}\neq 0}=\Pr{Z^r_i\neq Z^r_j}=\mu^r(x_i\neq x_j).
\end{equation} 
and
\begin{equation}
\label{dummyEq58}
    \E_{\nu}\left[\Delta d_{n+1}\right]=\E\left[Z^r_i-Z^r_j\right]=q^r_{i}-q^r_{j}.
\end{equation}

If $\mu^r(x_i\neq x_j)\neq 0$, then the random walk is non-trivial, and in particular $\capdu{n=n_0}{\infty }\{Y^i_n=Y^j_n\}=\capdu{n=n_0}{\infty }\{d_n=0\}$ has probability $0$. If $r(i)=r(j)$, this means that $\capdu{n=n_0}{\infty }\{rk(Y_n)=r\}$ has probability $0$, contradicting the fact that $\Prd{\nu}{A}>0$. We conclude that if $r(i)=r(j)$, then $\mu^r(x_i\neq x_j)=0$.

For the second assertion, assume that in addition to $\mu^r(x_i\neq x_j)\neq 0$, we also have $q^r_{i}\leq q^r_j$. This means that either $d_n\to -\infty $ or the random walk is recurrent. In either case, $\Prd{\nu}{\capdu{n=n_0}{\infty }\{Y^i_n>Y^j_n\}}=\Prd{\nu}{\capdu{n=n_0}{\infty }\{d_n>0\}}=0$. Therefore, if $r(i)<r(j)$, then $\Prd{\nu}{\capdu{n=n_0}{\infty }\{rk(Y_n)=r\}}=0$, again contradicting the fact that $\Prd{\nu}{A}>0$. We conclude that if $r(i)<r(j)$, then either $\mu^r(x_i\neq x_j)=0$ or $q^r_i>q^r_j$.
\end{proof}

For the sufficiency part of \cref{theoremTerminalRankings}, we are going to prove the following more general result.

\begin{lemma}[Terminal rankings sufficient condition]
\label{theoremLimitRankingsSufficient}
Let $r\in {\cal R}$ and define $A=\{(i,j)\in {\cal R}\times {\cal R}:r(i)<r(j)\}$ and $A'=\{(i,j)\in {\cal R}\times {\cal R}:r(i)=r(j)\}$. Assume that for any $(i,j)\in A$, either $q^r_i>q^r_j$ or $\mu^r(x_i\neq x_j)=0$, and that for any $(i,j)\in A'$, $\mu^r(x_i\neq x_j)=0$. Then, there exists some $M>0$, such that for any initial distribution $\nu$ and any $n_0\in \N $ that satisfy
\begin{equation}
\label{conditionLargeDifferences}
     \Prd{\nu }{\capdu{(i,j)\in A}{} \left\{ X^{i}_{n_0}>X^{j}_{n_0}+M\right\},\capdu{(i,j)\in A'}{}\left\{ X^{i}_{n_0}=X^{j}_{n_0}\right\}}>0,
\end{equation}
we have
\begin{equation}
    \Prd{\nu }{\capdu{n=n_0}{\infty }\left\{ rk(X_n)=r\right\}}>0.
\end{equation}
\end{lemma}

\begin{proof}
Consider the collection of random variables $\{U^i_n\}_{n\in\N}^{i\in [d]}$, independent of ${\cal F}_{\infty}$, such that for each $i$, $\{U^i_n\}_n$ are i.i.d. with distribution same as $Z^r_i$. For any pair $(i,j)\in A$, $U^i_n-U^j_n$ is either identically zero (if $\mu^r(x_i\neq x_j)=0$) or it has positive mean and finite variance (if $q^r_i-q^r_j>0$). In the latter case, by the Strong Law of Large Numbers, $\sumdu{m=1}{n}(U^i_m-U^j_m)\to \infty $ a.s. as $n\to \infty $. Therefore, in both cases, $\sumdu{m=1}{n}(U^i_m-U^j_m)$ is bounded below a.s. Hence, there exists some $M>0$, such that for any pair $(i,j)\in A$,
\begin{equation}
\label{eqDummy27}
    \Pr{\underset{n\in\N}{\min }\, \sumdu{m=1}{n}(U^i_m-U^j_m)\leq -M}<\frac{1}{d^2}.
\end{equation}
Now let the initial distribution $\nu$ and $n_0\in\N$ satisfy \cref{conditionLargeDifferences} for the value of $M$ specified in \cref{eqDummy27}, i.e. $\Prd{\nu }{D} >0$, where 
\begin{equation}
\begin{aligned}
\label{dummyEq29}
    D & =\capdu{(i,j)\in A}{} \left\{ X^{i}_{n_0}>X^{j}_{n_0}+M\right\}\cap \capdu{(i,j)\in A'}{} \left\{ X^{i}_{n_0}=X^{j}_{n_0}\right\}
\end{aligned}
\end{equation}
We want to show that $\Prd{\nu}{\capdu{n=n_0}{\infty }rk(X_n)=r}>0$. Let $Y_n$ be a process that $(r,n_0)$-mimics $X_n$ (see \cref{lemmaMirroring}) and note that \cref{dummyEq29} implies
\begin{equation}
\begin{aligned}
    D=\capdu{(i,j)\in A}{} \left\{ Y^{i}_{n_0}>Y^{j}_{n_0}+M\right\}\cap \capdu{(i,j)\in A'}{} \left\{ Y^{i}_{n_0}=Y^{j}_{n_0}\right\}.
\end{aligned}
\end{equation}
For any $(i,j)\in A'$, $n\geq n_0$, we have $\Prd{\nu}{\Delta Y^i_{n+1}\neq \Delta Y^j_{n+1}}=\mu^r(x_i\neq x_j)=0$ by \cref{lemmaMirroring}ii and by assumption, hence on the set $D$ we have
\begin{equation}
\label{eqDummy41}
    Y^i_n=Y^j_n,\text{ for all }(i,j)\in A',n\geq n_0.
\end{equation}
We further define
\begin{equation}
\begin{aligned}
    B_{(i,j)} & =\capdu{n\geq n_0}{}\left\{ (Y^i_{n}-Y^j_{n})-(Y^i_{n_0}-Y^j_{n_0})>-M\right\},\ \  (i,j)\in A,\\
    B & =\capdu{(i,j)\in A}{}\,B_{(i,j)}.
\end{aligned}
\end{equation}
Note that on the set $D\cap B_{(i,j)}$ we have $Y^i_n>Y^j_n$ for all $n\geq n_0$ and any $(i,j)\in A$. Combining this with \cref{eqDummy41}, we get that on the set $D\cap B$, it holds that $rk(Y_n)=r$ for all $n\geq n_0$, which implies $rk(X_n)=r$ for all $n\geq n_0$ (\cref{lemmaMirroring}iii). It is therefore enough to show that $\Prd{\nu}{D\cap B}>0$.

Note that for each $(i,j)\in A$, $\{ (Y^i_{n+n_0}-Y^j_{n+n_0})-(Y^i_{n_0}-Y^j_{n_0})\} _{n\in\N}$ has the same distribution as $\left\{\sumdu{m=1}{n}(U^i_m-U^j_m)\right\}_{n\in\N}$, therefore \cref{eqDummy27} implies that $\Prd{\nu }{B_{(i,j)}}>1-1\slash d^2$, and since $card(A)<d^2$, we get $\Prd{\nu }{B}>0$. By assumption we also have $\Prd{\nu}{D}>0$. Finally observe that by \cref{lemmaMirroring}ii, $D\in {\cal G}_{n_0}$ and $B\independent {\cal G}_{n_0}$, hence $\Prd{\nu}{D\cap B}=\Prd{\nu}{D}\cdot \Prd{\nu}{B}>0$, which completes the proof.
\end{proof}

\begin{proof}[Proof of sufficiency for \cref{theoremTerminalRankings}]
By assumption $r$ satisfies the conditions of \cref{theoremLimitRankingsSufficient}. Let $M$ be as in that lemma and define the initial distribution $\nu$ as follows: $X^i_0=(d-r(i))\cdot (M+1)$ a.s. Then, $r(i)=r(j)$ implies $X^i_0=X^j_0$ a.s., while $r(i)<r(j)$ implies $X^i_0>X^j_0+M$ a.s. That is, $\nu$ satisfies \cref{conditionLargeDifferences} with $n_0=0$, hence $\Prd{\nu}{\capdu{n=0}{\infty}\{ rk(X_n)=r\}}>0$, in particular $r$ is terminal.

\end{proof}

\subsection{Limit theorems for $X_n$}
\label{sectionLimitTheorems}

In this section we will prove the following theorem about the long term behavior of $X_n$.

\begin{proposition}[Strong Law of Large Numbers and Central Limit Theorem]
\label{marketShareTheorem}
For any $r\in {\cal R}$,
\begin{equation}
\label{eqSLLN}
    \underset{n\to \infty }{\lim }\frac{X_n}{n}=q^{r}\ \text{ a.s. on the set }\left\{ \underset{n\to \infty }{\lim} rk(X_n)=r\right\}.
\end{equation}
Furthermore, for any initial distribution $\nu$, if $\Prd{\nu}{\underset{k\to\infty }{\lim}rk(X_k)=r}>0$, then for each $i\in [d]$, $r\in {\cal R}$, and $x\in\R$,
\begin{equation}
\label{eqCLT}
    \limn \Prd{\nu}{\frac{X^i_n-n\cdot q^r_i}{\sqrt{n}\cdot \sigma ^r_i}\leq x\midvert \underset{k\to \infty }{\lim} rk(X_k)=r}=\Phi(x),
\end{equation}
where $\Phi$ denotes the cumulative distribution function of a standard normal distribution.
\end{proposition}

For the proof we are going to need a couple of lemmas whose proofs are given in the Appendix.

\begin{lemma}
\label{lemmaConditioningContinuity}
Let $A_n$, $n\in\N$, and $A$ be measurable sets in a probability space, each with positive probability, and suppose that $A_n\to A$ a.s. (i.e. $\Pr{(A_n\backslash A)\cup (A\backslash A_n)}\to 0$). Then, $\Pr{S\midvert A_n}\to \Pr{S\midvert A}$ uniformly in $S\in {\cal F}$.
\end{lemma}

\begin{lemma}
\label{lemmaDoubleConvergence}
Let $a_{m,n}\in\R$, $m,n\in\N$, and suppose that $\underset{m\to\infty}{\lim }a_{m,n}=a_n\in \R$ uniformly in $n$, and $\limn a_{m,n}=a\in\R$ for all $m\in\N$. Then, $\limn a_n=a$.
\end{lemma}

\begin{proof}[Proof of \cref{marketShareTheorem}]
Since by definition $\{rk(X_k)\to\infty \}=\cupdu{n_0\in\N}{}\,\capdu{k=n_0}{}\{rk(X_k)=r\}$, it is enough to show that 
\begin{equation}
\label{eqSLLNdummy}
    \underset{n\to \infty }{\lim }\frac{X_n}{n}=q^{r}\ \text{ a.s. on the set }{\capdu{k=n_0}{}\{rk(X_k)=r\}}
\end{equation}
and
\begin{equation}
\label{eqCLTdummy}
    \limn \Prd{\nu}{\frac{X^i_n-n\cdot q^r_i}{\sqrt{n}\cdot \sigma ^r_i}\leq x\midvert \capdu{k=n_0}{\infty }\{rk(X_k)=r\}}=\Phi(x),\ \ i\in[d],
\end{equation}
for any $n_0\in\N$ that satisfies $\Prd{\nu}{\capdu{k=n_0}{\infty }\{rk(X_k)=r\}}>0$. 

Fix such an $n_0$ and let $Y_n$ be a process that $(r,n_0)$-mimics $X_n$ (see \cref{lemmaMirroring}). Since $\{\Delta Y_n\}_{n\geq n_0+1}$ is an i.i.d. sequence whose $i$-th component has mean $q^r_i$ and standard deviation $\sigma ^r_i$, we have by the Strong Law of Large Numbers,
\begin{equation}
\label{dummyEq20}
    \frac{Y_n}{n}\to q^r\text{ a.s.}
\end{equation}
and by the Central Limit Theorem,
\begin{equation}
\label{dummyEq47}
    \frac{Y^i_n-n\cdot q^r_i}{\sqrt{n}\cdot \sigma ^r_i}\overset{d}{\to }{\cal N}(0,1).
\end{equation}
\Cref{eqSLLNdummy} follows from \cref{dummyEq20} and \cref{lemmaMirroring}iii. To show \cref{eqCLTdummy}, first note that \cref{dummyEq47} can be strengthened: since $\Delta Y_{n+1}\independent {\cal F}_n$ for each $n\geq n_0$, we have that for any $m\geq n_0$ and any $x\in\R$,
\begin{equation}
\label{dummyEq23}
    \Prd{\nu}{\frac{Y^i_n-n\cdot q^r_i}{\sqrt{n}\cdot \sigma ^r_i}\leq x\midvert \capdu{k=n_0}{m}\{rk(X_k)=r\}}\overset{n\to\infty}{\to} \Phi (x).
\end{equation}
Furthermore, since $\capdu{k=n_0}{m}\{rk(X_k)=r\}\overset{m\to\infty }{\longrightarrow}\capdu{k=n_0}{\infty }\{rk(X_k)=r\}$, \cref{lemmaConditioningContinuity} implies that
\begin{equation}
\begin{aligned}
    \Prd{\nu}{\frac{Y^i_n-n\cdot q^r_i}{\sqrt n\cdot \sigma ^r_i}\leq x\midvert \capdu{k=n_0}{m}\{rk(X_k)=r\}}\overset{m\to\infty}{\to} \\
    \Prd{\nu}{\frac{Y^i_n-n\cdot q^r_i}{\sqrt n\cdot \sigma ^r_i}\leq x\midvert \capdu{k=n_0}{\infty }\{rk(X_k)=r\} }
\end{aligned}
\end{equation}
uniformly in $n$. Combining this with \cref{dummyEq23} and \cref{lemmaDoubleConvergence} we get
\begin{equation}
    \Prd{\nu}{\frac{Y^i_n-n\cdot q^r_i}{\sqrt n\cdot \sigma ^r_i}\leq x\midvert \capdu{k=n_0}{\infty }\{rk(X_k)=r\}}\to \Phi(x) ,
\end{equation}
as $n\to \infty$. By \cref{lemmaMirroring}iii, this is equivalent to \cref{eqCLTdummy}.

\end{proof}

We also have the following partial converse of \cref{marketShareTheorem}.

\begin{proposition}
\label{marketShareConverse}
If $X_n\slash n\to x\in \R ^d$ and the components of $x$ are all distinct, then $rk(X_n)\to rk(x)$ and $x=q^{rk(x)}$.
\end{proposition}
\begin{proof}
Let $i,j\in [d]$ and assume without loss of generality that $x_i>x_j$. Then, for large enough $n$, $X^i_n>X^j_n$, so $rk(X_n)$ ranks $i$ higher than $j$. Since this is true for all pairs $i,j$, we get that for large enough $n$, $rk(X_n)=rk(x)$, hence $rk(X_n)\to rk(x)$. By \cref{marketShareTheorem}, $X_n\slash n\to q^{rk(x)}$.
\end{proof}

\subsection{Terminal rankings and initial distributions}
\label{sectionInitialDistributions}

Although \cref{theoremTerminalRankings} gives the possible limits of the ranking for a ranking-based process in principle, it doesn't say for which pairs of initial distributions $\nu$ and terminal rankings $r$ we have $\Prd{\nu}{rk(X_n)\to r}>0$. To see that for the same terminal ranking $r$ it is possible to have $\Prd{\nu}{rk(X_n)\to r}>0$ for some initial distributions $\nu$ and not for others, consider a deterministic system with $d=2$ and such that
\begin{equation}
\begin{aligned}
    \Pr{\Delta X_{n+1}=(1,0)\midvert rk(X_n)=\text{id}_2} & =1 \ \ \text{ and }\\
    \Pr{\Delta X_{n+1} =(0,1)\midvert rk(X_n)\neq \text{id}_2} & =1,
\end{aligned}
\end{equation}
where $\text{id}_2$ is the identity function on the set $\{1,2\}$.

In words, if $X^1_n>X^2_n$, then $X^1_n$ increases by $1$ and $X^2_n$ remains constant. If $X^1_n\leq X^2_n$, then $X^2_n$ increases by $1$ and $X^1_n$ remains constant. Clearly, if we start at $X_0=(0,0)$, $rk(X_n)\to r$ a.s., where $r(1)=2$, $r(2)=1$, while if we start at $X_0=(1,0)$, $rk(X_n)\to \text{id}_2$ a.s.

From the above example it might seem that the only reason that a strict ranking $r$ satisfying $q^{r}_{r^{-1}(1)}>q^{r}_{r^{-1}(2)}>\ldots >q^{r}_{r^{-1}(d)}$ might fail to satisfy $\Prd{\nu}{rk(X_n)\to r}>0$ is that it is not reachable from the given initial distribution, in the sense that $\Pr[\nu]{\bigcup_{n=1}^{\infty }\{rk(X_n)=r\}}=0$. However, this is not the only case. For example, let $d=3$, and suppose that
\begin{equation}
\begin{aligned}
    \Pr{\Delta X_{n+1}=(5,-2,0)\midvert rk(X_n)=\text{id}_3} & =1\slash 2,\\
    \Pr{\Delta X_{n+1} =(-3,3,0)\midvert rk(X_n)=\text{id}_3} & =1\slash 2,\ \ \text{ and}\\
    \Pr{\Delta X_{n+1}=(0,0,1)\midvert rk(X_n)\neq \text{id}_3} & =1,
\end{aligned}
\end{equation}
In words, whenever $X^1_n>X^2_n>X^3_n$, with probability $1\slash 2$ the first component will increase by $5$ and the second will decrease by $2$, and also with probability $1\slash 2$ the first component will decrease by $3$ and the second will increase by $3$, while the last component remains constant a.s. For any other ranking, the third component increases by $1$ and the rest remain constant a.s.

Now suppose we begin at $X_0=(2,1,0)$ a.s., so that $rk(X_0)=\text{id}_3$ a.s. Clearly, after the first step the ranking will necessarily change and after that $\Delta X_{n}=(0,0,1)$ deterministically, so that for large $n$ we will have either $X^3_n>X^1_n>X^2_n$ or $X^3_n>X^2_n>X^1_n$. We see that despite the fact that $rk(X_0)=\text{id}_3$ and $q^{\text{id}_3}_{1}>q^{\text{id}_3}_2>q^{\text{id}_3}_3$, for the specific initial distribution $\nu$ we get $\Prd{\nu }{rk(X_n)\to {\text{id}_3}}=0$.

The above examples might seem discouraging. We have the following positive result, which states that such situations do not arise if a certain condition is satisfied. The condition roughly says that, no matter the ranking, there is some positive probability for any component to increase faster than the rest, and the increments of the rest to follow any given non-strict order.

\begin{proposition}
\label{propositionLimitRankingsSufficient}
Suppose that for any permutation $\sigma $ of $[d]$ and any $r'\in R$,
\begin{equation}
\label{dummyEq21}
    \mu^{r'}(x_{\sigma _1}>x_{\sigma _2}\geq x_{\sigma _3}\geq\ldots\geq x_{\sigma _d})>0.
\end{equation}
Then, for any initial distribution $\nu $ and any terminal ranking $r$,
\begin{equation}
    \Prd{\nu}{rk(X_n)\to r}>0.
\end{equation}
\end{proposition}

\begin{remark}
\label{remarkOnlyStrictRankings}
The condition of \cref{propositionLimitRankingsSufficient} implies that $\mu^r(x_i\neq x_j)>0$ for all $i,j\in [d]$ and $r\in {\cal R}$, which in particular implies the condition of \cref{theoremTerminalRankingsSpecial}. Consequently, under the condition of \cref{propositionLimitRankingsSufficient}, only strict rankings may be terminal.
\end{remark}

\begin{example}
In a ranking-based P\'olya urn, with probability one, exactly one of the components of $\Delta X_{n+1}$ is $1$ and the rest are $0$ (see also \cref{sectionPolyaUrn}). Therefore, \cref{dummyEq21} is satisfied if and only if for any ranking there is positive probability of adding a ball of any given color. In \cref{example:additivePolyaUrn}, this is equivalent to either $\lambda _d>0$ or $a_i>0$ for all $i\in [d]$.

More generally, for processes that change one component at a time, \cref{dummyEq21} is satisfied if and only if, for any ranking, every component has non-zero probability of increasing.
\end{example}

\begin{proof}[Proof of \cref{propositionLimitRankingsSufficient}]
By \cref{remarkOnlyStrictRankings} we may assume that $r$ is a strict ranking. Also, by renaming the indices, we may assume that $r$ is the identity map on $[d]$, i.e. $r(i)=i$ for all $i\in [d]$. Let $M>0$ be as in \cref{theoremLimitRankingsSufficient} and define
\begin{equation}
\begin{aligned}
    C^j_{n} & =\left\{X^j_{n}>X^{j+1}_{n}+M\right\},\ \ j\in[d-1],\ \  n\in\N,\\
    B^i_n & =\capdu{j=i}{d-1}C^j_{n},\ \ i\in[d-1],\ \  n\in\N,
\end{aligned}
\end{equation}
and $B^d_n=\Omega $, $n\in\N$. By \cref{theoremLimitRankingsSufficient}, it is enough to show that $\Prd{\nu}{\cupdu{n\in\N}{}B^1_n}>0$. We will use (backwards) induction on $i$ to show that $\Prd{\nu}{\cupdu{n\in\N}{}B^i_n}>0$ for all $i\leq d$, with the base case $i=d$ being trivially true. Suppose then that $\Prd{\nu}{\cupdu{n\in\N}{}B^{i+1}_n}>0$ or, equivalently, that there exists some $n\in\N$ such that $\Prd{\nu}{B^{i+1}_n}>0$. Fix such an $n$. From \cref{dummyEq21} and continuity, there exists some $\epsilon>0$ such that $\mu^{r'}(A_i)>0$ for all $r'\in{\cal R}$, where
\begin{equation}
    A_i=\{x\in\R^{d}:x_i-\epsilon >x_{i+1}\geq x_{i+2}\geq \ldots \geq x_{d}\}.
\end{equation}
For any $j\in [d-1]$ and $k\in\N$, define 
\begin{equation}
D^j_{m,k}=\left\{X^j_{m+k}-X^j_{m}\geq X^{j+1}_{m+k}-X^{j+1}_{m}\right\}.
\end{equation}
and 
\begin{equation}
D^j_{m,k}(\epsilon)=\left\{X^j_{m+k}-X^j_{m}-\epsilon \geq X^{j+1}_{m+k}-X^{j+1}_{m}\right\}.
\end{equation}
In particular, $D^j_{m,1}=\left\{\Delta X^j_{m+1}\geq \Delta X^{j+1}_{m+1}\right\}$, and similarly for $D^j_{m,1}(\epsilon)$. Therefore, from \cref{eqCondInd1} we get that for any $m\in\N$,
\begin{equation}
\label{dummyEq81}
\begin{aligned}
    \Prd{\nu}{D^i_{m,1}(\epsilon),\capdu{j=i+1}{d-1}D^j_{m,1}\midvert {\cal F}_m} & =\mu^{rK(X_m)}(A_i)\\
    & \geq \underset{r'\in {\cal R}}{\min }\, \mu^{r'}(A_i)>0\text{ a.s.}
\end{aligned}
\end{equation}

Let $K\in\N$ be such that
\begin{equation}
\label{dummyEq45}
    \Pr[\nu]{F_{K}\cap B^{i+1}_n}>0,
\end{equation}
where
\begin{equation}
    F_{K}=\{X^i_n-X^{i+1}_n>M-K\epsilon \}.
\end{equation}
This is always possible, since $\cupdu{K\in\N}{}F_{K}=\Omega $ and $\Pr[\nu]{B^{i+1}_n}>0$ by assumption. Applying \cref{dummyEq81} for $m=n,n+1,\ldots ,n+(K-1)$ and using $D^j_{m,1}\in{\cal F}_{m+1}$, it easily follows that
\begin{equation}
\label{dummyEq79}
\begin{aligned}
    \Prd{\nu}{D^i_{n,K}(K\epsilon),\capdu{j=i+1}{d-1}D^j_{n,K}\midvert {\cal F}_n}>0\ \ a.s.
\end{aligned}
\end{equation}

Observe that
\begin{equation}
\begin{aligned}
    C^j_n\cap D^j_{n,K}\subset C^j_{n+K}, & \ \ j=i+1,\ldots ,d-1\\
    F_{K}\cap D^i_{n,K}(K\epsilon)\subset C^i_{n+K}. &
\end{aligned}
\end{equation}
Combining these two relations and the definition of $B^i_n$ we get
\begin{equation}
\begin{aligned}
    \Pr[\nu]{B^i_{n+K}} & =\Pr[\nu]{\capdu{j=i}{d-1}C^j_{n+K}}\\
    & \geq \Pr[\nu]{\capdu{j=i+1}{d-1}\left(C^j_n\cap D^j_{n,K}\right), F_{K},D^i_{n,K}(K\epsilon)}\\
    & =\Pr[\nu]{F_{K},B^{i+1}_n,D^i_{n,K}(K\epsilon),\capdu{j=i+1}{d-1}D^j_{n,K}}\\
    & =\Pr[\nu]{F_{K},B^{i+1}_n}\cdot \Pr[\nu]{D^i_{n,K}(K\epsilon),\capdu{j=i+1}{d-1}D^j_{n,K}\midvert F_{K},B^{i+1}_n}\\
    & >0,
\end{aligned}
\end{equation}
with the last line following from \cref{dummyEq45,dummyEq79}. This concludes the inductive proof.
\end{proof}

\section{Applications}
\label{sectionApplications}

\subsection{Ranking-based P\'olya urns and urn functions}
\label{sectionPolyaUrn}

In this section we look at how our results apply to the case of ranking-based P\'olya urns in terms of urn functions. We call \textit{ranking-based P\'olya urn} a ranking-based process $X_n\in\R^d$ where $\Delta X_n\in \{ 0,1\}^d$ and $\sumdu{i=1}{d}\Delta X^i_n=1$ a.s. Note that in this case
\begin{equation}
    q^{rk(X_n)}_i=\E \left[\Delta X^i_{n+1}\midvert rk(X_n)\right]=\Pr{\Delta X^i_{n+1}=1\midvert rk(X_n)},
\end{equation}
that is, $q^r_i$ is the probability of adding a ball of color $i$, when the ranking is $r$.

We want to compare our results to \citep{arthur1986strong,hill1980strong}, where the results are stated in terms of the fixed points of the urn function. The urn function $f:\Delta ^{d-1}\to \Delta ^{d-1}$, where
\begin{equation}
    \Delta ^{d-1}:= \left\{ x\in [0,1]^d,\ \sumdu{i}{}\,  x_i=1\right\}
\end{equation}
is the standard $(d-1)$-dimensional simplex, takes as argument the vector of proportions of balls of each color, and its $i$-th component $f_i$ gives the probability of the next ball being of color $i$. For a ranking-based urn, $f(x)$ must be constant in regions of constant ranking, that is, its value may only depend on $rk(x)$. With our notation we have
\begin{equation}
\label{eqPolyaProba}
    f_i(x)=\Pr{\Delta X^i_n=1\midvert rk(X_n)=rk(x)}=q^{rk(x)}_i.
\end{equation}
The next proposition uses our results from \cref{sectionResultsGeneral} to relate the fixed points of $f$ with the limiting behavior of $X_n\slash n$.

\begin{proposition}
\label{propositionPolya}
Consider a ranking-based P\'olya urn with urn function $f$ and let $A$ be the set of fixed points of $f$ whose coordinates are all distinct, i.e.
\begin{equation}
    A=\{ x\in \Delta ^{d-1}: f(x)=x,\ x_i\neq x_j\ \text{ for all } i\neq j\}.
\end{equation}
Then:
\begin{enumerate}
    \item For any $x\in A$, there is some $\nu$ such that $\Prd{\nu}{X_n\slash n\to x}>0$. If $q^r_i>0$ for all $i\in [d],r\in{\cal R}$, then $\Prd{\nu}{X_n\slash n\to x}>0$ holds for all $\nu$.
    
    \item If $q^r_i>0$ for all $i\in [d],r\in{\cal R}$ and furthermore \cref{qualityOrderingAssumption} is satisfied, then for any initial distribution $\nu$, $\limn X_n\slash n\in A$ a.s. (in particular $X_n\slash n$ converges a.s.).
    
    \item Conditioned on $\limn X_n\slash n=x\in A$, $\frac{X^i_n-nx_i}{\sqrt{nx_i(1-x_i)}}$ converges to a standard normal distribution. More precisely, for any initial distribution $\nu$, $i\in[d]$, $x\in A$, and $y\in\R$,
    \begin{equation}
        \Prd{\nu}{\frac{X^i_n-n\cdot x_i}{\sqrt{nx_i(1-x_i)}}\leq y\midvert \underset{k\to\infty }{\lim }\frac{X_k}{k}=x}\to \Phi(y),
    \end{equation}
    whenever $\Pr[\nu]{\underset{k\to\infty }{\lim }\frac{X_k}{k}=x}>0$.
\end{enumerate}

\end{proposition}

\begin{proof}
\begin{enumerate}
    \item Let $x=(x_1,\ldots ,x_d)\in A$ and denote $r=rk(x)$, so that $x_{r^{-1}(1)}>\ldots >x_{r^{-1}(d)}$. Then, $f(x)=q^r$ (\cref{eqPolyaProba}). Since $x$ is a fixed point of $f$, we get $q^{r}=x$, hence also $q^{r}_{r^{-1}(1)}>\ldots >q^{r}_{r^{-1}(d)}$. By \cref{theoremTerminalRankings} $r$ is terminal, so $\Prd{\nu}{rk(X_n)\to r}>0$ for some initial distribution $\nu$. By \cref{marketShareTheorem}, $\Prd{\nu}{X_n\slash n\to x}=\Prd{\nu}{X_n\slash n\to q^r}\geq \Prd{\nu}{rk(X_n)\to r}>0$. If $q^r_i>0$ for all $i\in [d],r\in{\cal R}$, then the condition of \cref{propositionLimitRankingsSufficient} is satisfied, therefore $r$ being terminal implies $\Prd{\nu}{rk(X_n)\to r}>0$ for any initial distribution $\nu$.
    
    \item By \cref{settlingTheorem} $rk(X_n)$ converges a.s. and by \cref{theoremTerminalRankingsSpecial} the limit $R$ has to be a strict ranking, in particular $q^{R}_i\neq q^{R}_j$ for all $i\neq j$ a.s. By \cref{marketShareTheorem} $X^i_n\slash n\to q^{R}$ and by \cref{marketShareConverse} $q^{R}=q^{rk(q^{R})}$, which is a fixed point of $f$ by \cref{eqPolyaProba}, thus $q^{R}\in A$.
    
    \item Denote $r=rk(x)$. By \cref{marketShareConverse,marketShareTheorem}, $x=q^r$ and 
    \begin{equation}
        \left\{\underset{k\to\infty }{\lim }X_k\slash k=x\right\}=\left\{\underset{k\to\infty }{\lim }rk(X_k)=r\right\}\ \  a.s.
    \end{equation}
    Hence, by the second part of \cref{marketShareTheorem},
     \begin{equation}
        \Prd{\nu}{\frac{X^i_n-n\cdot x_i}{\sqrt{n}\cdot \sigma ^r_i}\leq y\midvert \underset{k\to\infty }{\lim }\frac{X_k}{k}=x}\to \Phi(y).
    \end{equation}
    The result follows once we recall that $\sigma ^r_i=\sqrt{Var(Z^r_i)}$ and that $Z^r_i$ is a Bernoulli random variable with parameter $q^r_i=x_i$, hence $\sigma ^r_i=\sqrt{x_i(1-x_i)}$.
\end{enumerate}
\end{proof}

We now compare our results to the ones that appear in \citep{arthur1986strong,hill1980strong}. We are going to restrict ourselves to ranking-based P\'olya urns with the urn function being constant in $n\in\N$ (in \citep{arthur1986strong} the urn function is allowed to be a function of $n$).

Part 1 of \cref{propositionPolya}, in particular the case $q^r_i>0$ for all $i\in [d],r\in{\cal R}$, agrees with Theorem 5.1 in \citep{arthur1986strong}. In that theorem, the authors show that $X_n\slash n$ has positive probability of converging to any point $\theta \in \Delta ^{d-1}$ that is a stable fixed point of $f$, in the sense that $f(\theta )=\theta $ and there is a neighborhood $U$ of $\theta $ and a positive-definite matrix $C$ such that
\begin{equation}
\label{eqStableFixedPoint}
    \left\langle C(x-f(x)),x-\theta)\right\rangle >0,\ \ \text{ for all }\ \  x\in \Delta ^{d-1}\cap U,\ \  x\neq \theta .
\end{equation}
Note that in the ranking-based case, where $f$ is piecewise constant, any fixed point $\theta $ with all coordinates being distinct (i.e. $\theta \in A$) is always stable, since then $f(x)=\theta $ identically in a neighborhood of $\theta $, so the above condition is satisfied if we take $C$ to be the identity matrix. The result in \citep{arthur1986strong} is more general than part 1 of \cref{propositionPolya}, because it also applies to fixed points whose coordinates are not distinct. On the other hand, there are no analogues of parts 2 and 3 of our \cref{propositionPolya} in \citep{arthur1986strong} that apply to the ranking-based case (but Theorem 3.1 in that reference is an analogue of part 2 for continuous urn functions $f$).

As mentioned in the introduction, in \citep{hill1980strong} the case $d=2$ is studied and it is shown that $X_n\slash n$ converges a.s. Note that we have shown this only if $q^r_i>0$ for all $i\in [d],r\in{\cal R}$, and \cref{qualityOrderingAssumption} is satisfied. In \citep{hill1980strong} no such assumption is made. However, the proof there relies on properties of the real line (when $d=2$, the process is described by $X^1_n$ alone, because $X^2_n=n-X^1_n$), thus it is not obvious how to generalize to $d\in\N$.

Regarding the support of the limit, Theorem 4.1 in \citep{hill1980strong} is similar to part 2 of our \cref{propositionPolya}: assuming that $d=2$ and $q^r_i>0$ for all $i\in [d],r\in {\cal R}$, if $A$ contains a single point, then the two results coincide. Part 2 of \cref{propositionPolya} also applies when $A$ contains more than one (i.e. two) points, while Theorem 4.1 in \citep{hill1980strong} does not. On the other hand, if $A$ is empty, which (in the case $d=2$ with $q^r_i>0$ for all $i\in [d],r\in {\cal R}$) is equivalent to \cref{qualityOrderingAssumption} not being satisfied, part 2 of \cref{propositionPolya} does not apply, while Theorem 4.1 in \citep{hill1980strong} implies that $X_n\slash n\to 1\slash 2$.

We emphasize that the above is a comparison of results in the special case of ranking-based P\'olya urns (and in the case of \citep{hill1980strong}, when $d=2$). However, both our results and those in \citep{arthur1986strong,hill1980strong} apply to more general settings: our results apply to more general (ranking-based) processes than P\'olya urns, while those in \citep{arthur1986strong,hill1980strong} apply to non-ranking-based P\'olya urns.

\subsection{Ranking items in online interfaces}
\label{sectionPopularityRanking}

A crucial setting where ranking-based reinforcement is common is online rank-ordered interfaces such as search engines, online marketplaces, newspapers and discussion forums. In this section we describe an application of our results to such systems. The model we describe is based on assumptions about the ranking algorithms implemented and user behavior, so we begin with motivating our assumptions. At the end of the section we describe how these assumptions may be relaxed.

Online interfaces often facilitate access to information for their users by ranking their content  \citep{liu2009learning}. People, in return, pay more attention to and interact more with results that appear higher on ranked lists \citep{germano2019few,joachims2005accurately}. One of the most fundamental and commonly employed ranking algorithms places the options on the screen according to their popularity, that is the number of clicks, sales, citations or upvotes that different options have obtained so far. The rank-by-popularity algorithm is very simple to implement, and many popular websites have relied on it in the past or use some version of it at present.\footnote{For example, Reddit used to order comments by the number of upvotes, Google scholar used to order articles by the number of citations (and still offers that possibility when looking at a profile), Amazon offers the possibility to order options by the number of reviews, Goodreads orders user comments by the number of likes, etc.} A wide array of behavioral models  about how people choose among different items in an ordered list have been postulated over the past years in economics, management, marketing and computer science (for a review of models in computer science see \cite{chuklin2015click}). We will consider a staple computer science model for the probability of clicking on a link, called the position-based model \cite[p. 10]{chuklin2015click}. We note that although we will refer to clicks, the model can also be used to describe downloads and citations of papers, purchases of products, likes of comments etc.

In the position-based model, a link is first examined by the user and then clicked if its content is considered to be relevant. This can be stated as
\begin{equation}
\label{dummyEq64}
    C^i_n=E^i_n\cap D^i_n,
\end{equation}
where
\begin{equation}
\begin{aligned}
    E^i_n& =\{ n\text{-th user examines link }i\},\\
    D^i_n& =\{\text{link } i\text{ is relevant to the } n \text{-th user}\},\\
    C^i_n & =\{ n\text{-th user clicks on link }i\},
\end{aligned}
\end{equation}

We are interested in the vector $X_n=\left(X^1_n,\ldots ,X^d_n\right)$, where $X^i_n$ is the number of users that have clicked on link $i$, up to the $n$-th user. Clearly, we have $\Delta X^i_{n+1}=1$ if $C^i_{n+1}$ occurs, and $\Delta X^i_{n+1}=0$ otherwise. Note that $X_n$ is not a ranking-based P\'olya urn as defined in \cref{sectionPolyaUrn}, because more than one of its components may change simultaneously.

The probability that a link is examined depends only on the position that it appears in, and typically decreases for later positions. Assuming that results appear according to the rank-by-popularity algorithm, that is, by descending number of clicks so far (and randomly breaking ties), this factor depends only on (a) the current rank of result $i$ with respect to the number of clicks and (b) the number of links that are ranked equally with it. For our purposes, we may allow the probability that a link is examined to depend on the full ranking (i.e. how all of the links are ranked), so we will denote
\begin{equation}
\label{dummyEq52}
    a^r_i=\Pr{E^i_{n+1}\midvert rk(X_n)=r}.
\end{equation}
The expression on the right hand side makes sense whenever $\{rk(X_n)=r\}$ has positive probability. We will be making this assumption below whenever similar expressions appear, without further mention.

We also assume that links that appear higher are more likely to be examined, that is, if $r(i)<r(j)$, then $a^r_i>a^r_j$. Finally, we assume that $a^r_i>0$ for all $i\in [d],r\in{\cal R}$, so that there is always positive probability of clicking on any of the links.

The probability of link $i$ being relevant to the user depends only on the link itself, that is $D^i_{n+1}$ is independent of $\left\{E^j_{n+1}\right\}_{j\in[d]},\left\{D^j_{n+1}\right\}_{j\neq i}$ and $rk(X_n)$. We denote
\begin{equation}
\begin{aligned}
\label{dummyEq66}
    u_i=\Pr{D^i_{n+1}}
\end{aligned}
\end{equation}
and assume that $u_i\in (0,1)$. 

The number $u_i$ can be considered a measure of objective quality of the link (not necessarily known to the ranking algorithm). Combining \cref{dummyEq64,dummyEq52,dummyEq66} we get
\begin{equation}
\begin{aligned}
\label{dummyEq50}
    q^r_i & =\Pr{\Delta X^i_{n+1}=1\midvert rk(X_n)=r}\\
    &=\Pr{C^i_{n+1}\midvert rk(X_n)=r}\\
    &=\Pr{E^i_{n+1}\cap D^i_{n+1}\midvert rk(X_n)=r}\\
    &=\Pr {E^i_{n+1}\midvert rk(X_n)=r}\cdot \Pr{D^i_{n+1}\midvert E^i_{n+1},rk(X_n)=r}\\
    &=\Pr{E^i_{n+1}\midvert rk(X_n)=r}\cdot \Pr{D^i_{n+1}}\\
    &=a^r_i\cdot u_i.
\end{aligned}
\end{equation}

Since we are assuming that $a^r_i>0$ for all $i\in [d],r\in{\cal R}$, and $u_i\in (0,1)$ for all $i\in [d]$, we also have $q^r_i\in (0,1)$ for all $i\in [d],r\in{\cal R}$. Moreover, using the fact that the $D^j_{n+1}$'s are independent of everything else and $\Pr{D^j_{n+1}}<1$ for all $j\in [d]$, we get
\begin{equation}
\label{eqDummy101}
\begin{aligned}
    &\ \ \ \ \Pr{\Delta X^i_{n+1}=1,\Delta X^j_{n+1}=0\ \text{ for all }j\neq i\midvert rk(X_n)=r}\\
    & \geq q^r_i\cdot \proddu{j\neq i}{}\Pr{\left(D^j_{n+1}\right)^c}>0,
\end{aligned}
\end{equation}
for any $i\in d,r\in{\cal R}$.

Now let $i\neq j$ and suppose (without loss of generality) that $u_i\geq u_j$. Recall that, by assumption, for any ranking $r$ that ranks $i$ higher than $j$, we have $a^r_i>a^r_j$, hence \cref{dummyEq50} gives $q^r_i>q^r_j$. That is, $i$ quasi-dominates $j$. By \cref{eqDummy101}, $\mu^r(x_i\neq x_j)>0$ for all $r\in {\cal R}$, therefore $i$ actually dominates $j$, hence \cref{qualityOrderingAssumption} is satisfied. 

\Cref{settlingTheorem} now says that $rk(X_n)$ converges a.s. \Cref{eqDummy101} also means that the conditions of \cref{theoremTerminalRankingsSpecial,propositionLimitRankingsSufficient} are satisfied, therefore the possible limits for $rk(X_n)$ are those strict rankings $r$ for which $q^r_{r^{-1}(1)}>\ldots >q^r_{r^{-1}(d)}$. Note that in general there will be more than one rankings $r$ satisfying this condition, especially if the effect of the position is strong ($a^r_i$ decreases quickly with the position of $i$ in the ranking $r$). Thus, it is likely that links of smaller objective quality $u_i$ will end up being ranked higher in the long-term (thus getting more clicks) than links of higher quality. This is an important consequence, since it implies that in general people will be directed towards links that are less likely to be relevant to them, and it reveals an inherent drawback of algorithms that rank results by popularity. 

Our framework can be generalized to other models of user behaviour. For example, we could allow the probability $a^r_i$ of examining a link to depend on the ranking in an arbitrary way (subject to \cref{qualityOrderingAssumption} being satisfied). In particular the model applies to cases where the position of other links also affects the probability of examining a link at a certain position, such as in the cascade model in computer science \citep{craswell2008experimental} or satisficing models in economics \citep{caplin2011search}. More generally, the assumption that links are first examined and then independently judged to be relevant or not can be discarded altogether; it is enough to require that the links possess some objective quality $u_i$, and whenever link $i$ is ranked higher than $j$ and $u_i>u_j$, it is more likely for $i$ to be clicked (i.e. $q^r_i>q^r_j$). For example, the $q^r_i$'s can be described by a multi-attribute utility model \citep{keeney1993decisions}, where the link position is one of the attributes and $u_i$ is a summary of the rest of the attributes.

In a similar vein, we can relax assumptions related to the ranking algorithm. For instance, more sophisticated ranking algorithms may not rank the links based on their number of clicks only, but according to some calculated score that takes into account several other features \citep{page1999pagerank,liu2009learning}. The conceptual framework we developed in this section still applies, as long as the popularity is taken into account in calculating the score. Further, recent algorithmic approaches estimate the objective utility or relevance $u_i$ of different items by debiasing the number of clicks from attention imbalances \citep{joachims2017unbiased,agarwal2019general}. Even for these algorithms, however, ranking-based rich-get-richer dynamics can be at play, if a link's actual or perceived utility for the users depends on the object's popularity \citep{arthur1989competing,muchnik2013social}. For example, when ranking social networking applications, the rank may convey information about their utility, therefore some form of advantage may persist even when correcting for attention disparities.

\section{Discussion}
We have developed a mathematical framework for describing systems characterized by ranking-based rich-get-richer dynamics. Specifically, we defined a ranking-based process as a discrete-time Markov process in $\R^d$ whose increment distributions depend only on the current ranking of the components of the process. Under a ranking-based reinforcement assumption (\cref{qualityOrderingAssumption}), we showed that the ranking converges (\cref{settlingTheorem}) and proved a Strong Law of Large Numbers and a Central Limit Theorem-type result for the process itself (\cref{marketShareTheorem}). We also found conditions in terms of the Markov transition kernel to check whether a particular ranking is a possible limit ranking. In some cases we were able to characterize the support of the limit of the ranking independently of the initial distribution (\cref{propositionLimitRankingsSufficient}). We also translated our results in terms of urn functions for the special case of ranking-based P\'olya urns, in order to compare them with previous results with which they partially overlap (\cref{sectionPolyaUrn}). Finally, we described an application to rank-ordered web interfaces (\cref{sectionPopularityRanking}).

Models of systems with rich-get-richer dynamics have been commonplace in the social, behavioral and computer sciences, and they have been used to describe the observed dynamics in a wide variety of settings. So far, there have been two main families of such models. The first family goes back to Gibrat's law \cite{gibrat1931inegalits}, which states that firms grow proportionally to their current size, and independently of the performance of their competitors. Variations of the notion of proportional growth have been applied across disciplines, for example to model citation growth \citep{allison1982cumulative} and city growth \citep{gabaix1999zipf,gabaix2008power}. Models based on Gibrat's law are inherently unsuitable for capturing ranking-based dynamics, because of their assumption that growth is independent of any competitors.

The second family of rich-get-richer models builds on the notion of preferential attachment \citep{simon1955class}, which assumes that entities grow when new units ``attach'' to them, but these new units are more likely to attach to entities that are already larger. Such models are usually described mathematically as P\'olya urns or one of their many generalizations \citep{mahmoud2008polya,pemantle2007survey}. What is common in almost all of these generalized P\'olya urns, and relevant to us, is the fact that the number of balls of a given color added is chosen from a finite set, with probabilities that are each a \textit{continuous} function of the proportion of balls of a \textit{single} color, except that they are normalized to sum to one. Although this allows for some form of competition among colors, it precludes direct comparison of the proportions of balls of different colors, so it does not allow modeling systems where the growth rates depend on the differences between the sizes of different entities, let alone their ranking. Two exceptions are the works of Arthur et al. \cite{arthur1986strong} and of Hill et al. \cite{hill1980strong}, which allow arbitrary comparisons of proportions of balls of different colors, but they only treat the simplest type of P\'olya urn processes. These works do not specifically focus on ranking-based competition, but they partially cover them as extreme cases, with a subset of their results applying to them. See the Introduction and \cref{sectionPolyaUrn} for details.

Compared to these existing approaches, our work differs in two main ways. First, our approach is at a more abstract level; the literature related to preferential attachment and Gibrat's law usually starts with a specific model, with the goal of reproducing some empirically observed phenomena, such as outcome unpredictability and skewed popularity distributions. Our approach in contrast is model-independent; we have identified conditions that are sufficient to lead to certain rich-get-richer phenomena, i.e. conditions that when satisfied by \textit{any} model, regardless of the exact assumptions made, they lead to the stated results. This is illustrated in \cref{sectionPopularityRanking}, where we point out that ranking-based rich-get-richer dynamics could be set in motion under a wide array of behavioral or algorithmic assumptions, as long as \cref{qualityOrderingAssumption} is satisfied. In this respect, our work is similar in spirit to the work of Arthur et al. \citep{arthur1986strong}.

The second and perhaps more distinctive difference of our work, is the fact that it covers an opposite end of the spectrum of rich-get-richer dynamics. The distributions of the increments of the various components, instead of depending (continuously) on the current level of each of the components separately, they are piecewise constant with respect to the current levels, with discontinuities occurring when the ranking of the components changes. In other words, we focus explicitly on the role of ranking-based competition. However, our framework does not consider other types of competition, nor does it allow for any explicit dependence of the increments on the current level of the process, other than through the ranking.

The above delineates a promising future research direction: one could envisage a general mathematical theory of Markov rich-get-richer processes that encompasses all of the above cases, by allowing for an arbitrary dependence of the increments' distribution on the current level of the whole vector of the process, subject to the minimal conditions for rich-get-richer dynamics. The work of Arthur et al. \cite{arthur1986strong} is in this direction for the case of simple P\'olya urn processes, but no such framework currently exists for more general processes.

\begin{appendix}

\section{Rankings are equivalent to weak orderings}
\label{appendixWeakOrderings}

The following proposition says that rankings are equivalent to weak orderings. A weak ordering on a set is like a total ordering, except that it allows for ``ties''. More precisely, a weak ordering ``$\succeq $'' on $S$ is a binary relation that is transitive and strongly complete, i.e. that for any two elements $a,b\in S$, at   least one of the relations $a\succeq b$ or $b\succeq a$ holds \cite{roberts1985measurement}. Recall that we would get a \textit{total} order, if we further required that $a\succeq b$ and $b\succeq a$ implies $a=b$.

\begin{proposition}
\label{propositionWeakOrderings}
    There is a bijection between rankings of a finite set $S$ and weak orderings on $S$, given by $r\mapsto \succeq _r$, where
    \begin{equation}
    \label{dummyEq61}
        a\succeq _r b \text{ whenever } r(a)\leq r(b).
    \end{equation}
    The above map satisfies
    \begin{equation}
    \label{dummyEq63}
        r(a)=card\{b\in S:a\nsucceq _r b\}+1.
    \end{equation}
    The ranking $r$ is strict if and only if $\succeq _r$ is a total order on $S$.
\end{proposition}
\begin{proof}
It is easy to check that $\succeq _r$, as defined by \cref{dummyEq61} is a weak ordering on $S$. Using \cref{dummyEq61}, \cref{dummyEq63} can be rewritten as
\begin{equation}
\label{dummyEq65}
    r(a)=card\{b\in S:r(b)<r(a)\}+1,
\end{equation}
which is equivalent to \cref{dummyEq67}, so it holds by definition. By \cref{dummyEq63}, $r$ is uniquely determined by $\succeq _r$, so the map $r\mapsto \succeq _r$ is one-to-one. To show that it is onto, let ``$\succeq$'' be a weak ordering on $S$ and define $r:S\to [S]$ by
\begin{equation}
\label{dummyEq69}
    r(a)=card\{b\in S:a\nsucceq b\}+1.
\end{equation}
We claim that $r(b)\leq r(a)$ is equivalent to $b\succeq a$. First note that if $b\succeq a$, then by transitivity $\{c\in S:b\nsucceq c\}$ is a subset of $\{c\in S:a\nsucceq c\}$, hence $r(b)\leq r(a)$. For the converse, assume that $b\nsucceq a$. Then we must have $a\succeq b$, and we get as above that $\{c\in S:a\nsucceq c\}$ is a subset of $\{c\in S:b\nsucceq c\}$, but this time it is a proper subset, because $a$ belongs to the latter. Therefore $r(a)<r(b)$, which completes the proof of our claim. Hence, by \cref{dummyEq61}, $\succeq$ is the same relation as $\succeq_r$, which shows that the mapping $r\mapsto \succeq_r$ is onto.

The last assertion follows from the fact that $a\succeq _r b$ and $b\succeq _r a$ hold simultaneously if and only if $r(a)=r(b)$.

\end{proof}

\section{Supporting proofs}
\label{appendix}

Here we give the proofs of \cref{lemmaBiasedRandomWalk,lemmaConditioningContinuity,lemmaDoubleConvergence}. For ease of reference, we repeat each statement before the proof.

\begin{replemma}{lemmaBiasedRandomWalk}
\label{lemmaBiasedRandomWalk2}
Let $(\Omega , {\cal G},\mathbb{P})$ be a probability space. Let $S$ be a finite set and for each $r\in S$, $\nu^r$ a distribution on $\R$ such that it either has positive mean or $\nu^r(\{0\})=1$. Let $\{ R_n\} _{n\in\N}$ be a sequence of random elements in $S$ and $\{ Y_n\} _{n\in \N }$ a sequence of random variables with $Y_0=0$. Suppose that $\Delta Y_{n+1}$ is conditionally independent of $\{(Y_k,R_k)\}_{k\leq n}$ conditioned on $R_n$, with distribution $\nu ^{R_n}$. In other words, for any $A\in {\cal B}(\R)$, $n\in\N$,
\begin{equation}
\label{eqDummy94}
    \Pr{\Delta Y_{n+1}\in A\midvert \{(Y_k,R_k)\}_{k\leq n}}=\nu^{R_n}(A)\ \ a.s.
\end{equation}
Then,
\begin{equation}
\label{eqDummy97}
    \Pr{\capdu{n\in\N}{}\{Y_n\geq 0\}}\geq \epsilon >0,
\end{equation}
where $\epsilon$ depends only on the distributions $\nu^r$, $r\in S$.
\end{replemma}

\begin{proof}
Let $\{U^r_n\}^{r\in S}_{n\in\N}$ be a collection of independent random variables, independent of $\{(Y_n,R_n)\}_{n\in\N}$, and such that $U^r_n\sim \nu^r$ for all $r\in S$, $n\in\N$, where the relation $\sim$ means equality in distribution. Define $Y'_0=Y_0=0$ and for each $n\in\N$, \begin{equation}
\label{eqDummy85}
    Y'_{n+1}=Y'_n+U^{R_n}_{n}.
\end{equation}
Clearly, for any $A\in{\cal B}(\R^d)$,
\begin{equation}
\label{eqDummy33}
\begin{aligned}
    \Pr{\Delta Y_{n+1}'\in A\midvert \{(Y'_k,R_k)\} _{k\leq n}} & = \Pr{U^{R_n}_{n}\in A\midvert R_n}=\nu^{R_n}(A),
\end{aligned}
\end{equation}
therefore $\{Y'_n\}_{n\in\N}\sim \{Y_n\}_{n\in\N}$. It is hence enough to show that \cref{eqDummy97} holds for the sequence $Y'_n$ instead of $Y_n$.

For each $r\in S$, define $\tau ^r_0=-1$ and inductively $\tau ^r_n=\inf\, \{k>\tau ^r_{n-1}:R_k=r\}$. Note that each $Y_n$ is a sum of terms of the form $U^r_{\tau _k}$, for $k=1,\ldots ,m_r$, where $m_r\in \N$, $r\in S$. Therefore,
\begin{equation}
\label{eqDummy99}
    \capdu{n\in\N}{}\{Y'_{n}\geq 0\}\supset\capdu{r\in S}{}\capdu{n\in\N}{}\left\{\sumdu{k=1}{n}U^{r}_{\tau ^r_k}\geq 0\right\},
\end{equation}
with the convention $U^{r}_{\tau ^r_k}=0$ when $\tau^r_k=\infty $.

Since $\tau^r_k\independent\{U^r_n\}_{n\in\N}$, if the $\tau ^r_k$'s were all finite a.s., it would easily follow that $U^r_{\tau ^r_k}$ has the same distribution as $U^r_1$, $k\in\N$, and since $\tau^r_k$ is strictly increasing in $k$ we would even get that $\{U^r_{\tau^r_k}\}_{k\in\N}$ is i.i.d. To deal with the case $\tau^r_n=\infty$, we define the random times $\sigma ^r_n$ as follows: Let $v^r=\sup\,\{n\in\N:\tau ^r_n<\infty \}$ and
\begin{equation}
    \sigma  ^r_n=\tau ^r_n\cdot \mathbf{1}_{n\leq v^r}+(\tau ^r_{v^r}+n-v_r)\cdot \mathbf{1}_{n>v^r}.
\end{equation}
The $\sigma ^r_n$'s are almost surely finite and distinct for fixed $r\in S$, and $\{\sigma ^r_n\}^{r\in S}_{n\in\N}\independent\{U^r_n\}^{r\in S}_{n\in\N}$. Therefore, by \cite[Theorem 2.1]{melfi2000estimation} we get that $\left\{U^r_{\sigma^r_n}\right\}^{r\in S}_{n\in\N}\sim \left\{U^r_{n}\right\}^{r\in S}_{n\in\N}$. (In \cite{melfi2000estimation} it is assumed that the $\sigma ^r_n$'s are all distinct a.s., even for different $r$'s, but this assumption can be substituted by the fact that the sequences $\{U^r_n\}_{n\in\N}$ are independent for different $r$'s and the proof goes through.)

Now observe that $\sigma ^r_k=\tau ^r_k$ on $\{\tau ^r_k<\infty \}$, therefore, by \cref{eqDummy99},
\begin{equation}
    \capdu{n\in\N}{}\{Y'_n\geq 0\}\supset \capdu{r\in S}{}\capdu{n\in\N}{}\left\{\sumdu{k=1}{n}U^r_{\sigma^r_k}\geq 0\right\}
\end{equation}
Consequently,
\begin{equation}
\begin{aligned}
    \Pr{\capdu{n\in\N}{}\{Y'_n\geq 0\}} & \geq \Pr{\capdu{r\in S}{}\capdu{n\in\N}{}\left\{\sumdu{k=1}{n}U^r_{\sigma^r_k}\geq 0\right\}}\\
    & =\proddu{r\in S}{}\Pr{\capdu{n\in\N}{}\left\{\sumdu{k=1}{n}U^r_{k}\geq 0\right\}}>0,
\end{aligned}
\end{equation}
because $\{U^r_k\}_{k\in\N}$ is an i.i.d. sequence of random variables that are either identically $0$ or they have a positive mean.
\end{proof}

\begin{replemma}{lemmaConditioningContinuity}
\label{lemmaConditioningContinuityAppendix}
Let $A_n$, $n\in\N$, and $A$ be measurable sets in a probability space, each with positive probability, and suppose that $A_n\to A$ a.s. (i.e. $\Pr{(A_n\backslash A)\cup (A\backslash A_n)}\to 0$). Then, $\Pr{S\midvert A_n}\to \Pr{S\midvert A}$ uniformly in $S\in {\cal F}$.
\end{replemma}

\begin{proof}
We have
\begin{equation}
\begin{aligned}
    &\hspace{0.47cm} \left|\Pr{S\midvert A}-\Pr{S\midvert A_n}\right|\\
    & =\left|\frac{\Pr{S\cap A}}{\Pr{A}}-\frac{\Pr{S\cap A_n}}{\Pr{A_n}}\right|\\
    &=\left|\frac{\Pr{S\cap A_n}+\Pr{S\cap A\backslash A_n}-\Pr{S\cap A_n\backslash A}}{\Pr{A}}-\frac{\Pr{S\cap A_n}}{\Pr{A_n}}\right|\\
    & \leq \Pr{S\cap A_n}\cdot \left|\frac{1}{\Pr{A}}-\frac{1}{\Pr{A_n}}\right|+\frac{\left|\Pr{S\cap A\backslash A_n}-\Pr{S\cap A_n\backslash A}\right|}{\Pr{A}}\\
    &\leq \left|\frac{1}{\Pr{A}}-\frac{1}{\Pr{A_n}}\right|+\frac{\Pr{A\backslash A_n}+\Pr{A_n\backslash A}}{\Pr{A}}\\
\end{aligned}
\end{equation}
The quantity in the last line does not depend on $S$ and, by assumption, it converges to $0$ as $n\to\infty $.
\end{proof}

\begin{replemma}{lemmaDoubleConvergence}
\label{lemmaDoubleConvergenceAppendix}
Let $a_{m,n}\in\R$, $m,n\in\N$, and suppose that $\underset{m\to\infty}{\lim }a_{m,n}=a_n\in \R$ uniformly in $n$, and $\limn a_{m,n}=a\in\R$ for all $m\in\N$. Then, $\limn a_n=a$.
\end{replemma}
\begin{proof}
Let $\epsilon >0$ and let $m_0\in\N$ be such that $|a_{m_0,n}-a_n|<\epsilon $ for all $n\in\N$. Now let $n_0\in\N$ be such that $|a_{m_0,n}-a|<\epsilon $ for all $n\geq n_0$. It follows that $|a_n-a|<2\epsilon$ for all $n\geq n_0$.
\end{proof}

\end{appendix}

\section*{Acknowledgements}

We would like to thank and Thorsten Joachims, Gabor Lugosi and Murad Taqqu for their remarks in previous versions of this manuscript. This research was supported in part
through NSF Award IIS-1513692.

\bibliographystyle{plainnat}
\bibliography{ranking-based-systems}

\end{document}